\documentclass{sapm}

\usepackage{graphicx}

\jname{STUDIES IN APPLIED MATHEMATICS}
\jvol{AA}
\jyear{YYYY}
\doi{10.1111/((please add article doi))}

\usepackage{amsmath}
\usepackage{amssymb}
\usepackage{esvect}

\newcommand\nonu{\nonumber}
\newcommand\sLP{\\[\smallskipamount]}
\newcommand\mLP{\\[\medskipamount]}
\newcommand\ZZ{\mathbb{Z}}
\newcommand\FSA{\mathcal{A}}
\newcommand\FSH{\mathcal{H}}
\newcommand{\DAHA}{\FSH\mkern-6mu\FSH}
\newcommand\FST{\mathcal{T}}
\newcommand\al\alpha
\newcommand\be\beta
\newcommand\ga\gamma
\newcommand\de\delta
\newcommand\ep\varepsilon
\newcommand\la\lambda
\newcommand\si\sigma
\newcommand\om\omega
\newcommand\La\Lambda
\newcommand\half{\frac12}
\newcommand\thalf{\tfrac12}
\newcommand\iy\infty
\newcommand\td{\tilde}
\newcommand\wh{\widehat}
\newcommand\wt{\widetilde}
\newcommand\lan{\langle}
\newcommand\ran{\rangle}
\newcommand\bma{\begin{pmatrix}}
\newcommand\ema{\end{pmatrix}}
\newcommand{\qhyp}[5]{\,\mbox{}_{#1}\phi_{#2}\!\left(
  \genfrac{}{}{0pt}{}{#3}{#4};#5\right)}
\newcommand\AW{{\rm AW}}

\begin{document}

\title{Dualities in the $q$-Askey scheme and degenerate DAHA}

\author[T. H. Koornwinder and M. Mazzocco]{Tom H.
Koornwinder\thanks{Korteweg-de Vries Institute for Mathematics,
University of Amsterdam, email: T.H.Koornwinder@uva.nl} and Marta
Mazzocco\thanks{School of Mathematics, The University of Birmingham,
email: m.mazzocco@bham.ac.uk}}
\affil{Korteweg-de Vries Institute for Mathematics,
University of Amsterdam, P.O. Box 94248, 1090 GE Amsterdam, Netherlands}
\affil{School of Mathematics, The University of Birmingham, Edgbaston,
Birmingham B15 2TT, United Kingdom}
\maketitle

\begin{abstract}
The Askey-Wilson polynomials are a four-parameter family of orthogonal
symmetric Laurent polynomials $R_n[z]$ which are eigenfunctions of a
second-order $q$-difference operator $L$, and of a second-order
difference operator in the variable $n$ with eigenvalue $z +
z^{-1}=2x$. Then $L$ and multiplication by $z+z^{-1}$ generate the
Askey-Wilson
(Zhedanov) algebra. A nice property of the Askey-Wilson polynomials is
that the variables $z$ and $n$ occur in the explicit expression
in a similar and to some extent exchangeable way.
This property is called duality. It returns in the non-symmetric case
and in the underlying algebraic structures:
the Askey-Wilson algebra and the double affine Hecke algebra (DAHA).
In this paper we follow the degeneration of the
Askey-Wilson polynomials until two arrows down and in four
different situations: for the orthogonal
polynomials themselves, for the degenerate Askey-Wilson algebras,
for the non-symmetric polynomials and
for the (degenerate) DAHA and its representations.
\end{abstract}

\section{Introduction}
The Askey-Wilson (briefly AW) polynomials \cite{AW} are a four-parameter
family of orthogonal polynomials which are eigenfunctions of
a second-order $q$-difference operator $L$, and which are explicitly
expressed as (terminating) basic hypergeometric series \cite{GR}.
We will write them as symmetric Laurent polynomials $R_n[z]$ of
degree $n$.
As orthogonal polynomials they satisfy a three-term recurrence relation.
In other words, $R_n[z]$ is also an eigenfunction with eigenvalue
$z+z^{-1}$ of a second-order difference operator in the variable $n$.
The operator $L$ and the operator of multiplication by $z+z^{-1}$,
both acting on symmetric Laurent polynomials $f[z]$, generate the
{\em Zhedanov} or {\em AW algebra} \cite{Zhe}, which can be
presented by generators and simple relations.

The idea of non-symmetric special functions, which yield (usually
orthogonal) symmetric special functions by symmetrization, started
with the introduction of the Dunkl operators \cite{Du89}, which
are differential-reflection operator associated with a root system.
These were generalized to Dunkl-Cherednik operators $Y$, which
are $q$-difference-reflection operators associated with root systems,
and which appear in the basic (or polynomial) representation of the
{\em double affine Hecke algebra} (DAHA) \cite{Cher}.
Non-symmetric Macdonald polynomials arose as eigenfunctions
of these operators $Y$. A more general DAHA \cite{Sa}
yielded non-symmetric Macdonald-Koornwinder polynomials.
In the rank one case this is the DAHA of type $(\check C_1,C_1)$
(the AW DAHA), which yields the non-symmetric AW polynomials \cite{NS}.
Furthermore, the AW algebra and the AW DAHA are closely connected.
A central extension of the AW algebra can be embedded
\cite{K}, \cite{Terw2} in the AW DAHA, while conversely the
AW algebra is isomorphic \cite{K1} to the spherical subalgebra of
the AW DAHA.

A nice property of AW polynomials $R_n[z]$, directly visible in the
$q$-hypergeometric expression, is that the variables $z$ and $n$
occur in a similar and to some extent exchangable way (completely
exchangable in the corresponding discrete $q$-Racah case
\cite[Section 14.2]{KLS} and in the case of the Askey-Wilson functions
\cite{KS2}, see also Remark \ref{160} and \S\ref{162}).
This property is called {\em duality}. It returns in the non-symmetric
case and in the underlying algebraic structures of the AW algebra and
the AW DAHA. Therefore the duality extends to the operators
occurring in the basic representations of these algebraic
structures and having the symmetric or non-symmetric AW polynomials
as eigenfunctions.

The AW polynomials are on top of the $q$-Askey scheme and
(by letting $q\to1$) the Askey scheme \cite[Ch.~9, 14]{KLS}.
The orthogonal polynomials in the ``lower'' families are limits
of AW polynomials (or $q$-Racah polynomials).
There are also \cite{GLZ}
corresponding limits of the AW algebra.
The duality property of the AW polynomials
then may also have a limit, but
usually as a duality between two different families.
In general, two different families of special functions
$\phi_\la(x)$ and $\psi_\mu(y)$, both occurring as eigenfunctions
of a certain operator, are dual when
$\phi_\la(x)=\psi_{\si(x)}(\tau(\la))$ for certain functions $\si$ and
$\tau$ (possibly only for a restricted set of values of $\la$ and $x$).
When the equality holds for all spectral values of the two operators
then this property is called {\it bispectrality}.
It was first emphasized in the context of differential
(rather than $q$-difference) operators
by Duistermaat and Gr\"unbaum \cite{DuisG}. In that seminal paper,
motivated by the need to analyze the relation between amount of data
and image quality in limited angle tomography, the authors classified
all possible potentials in the Schr\"odinger equation such that the
wave functions would be an eigenfunctionof a difference operator in the
spectral parameter as well. Since then, the bispectrality has been key for
the  determination of special solutions of the KdV equation
and of the KP hierarchy \cite{ZM}, \cite{Zu}, \cite{Wi} and 
of many other integrable equations including integrable systems of
particles \cite{Ka}, \cite{GH}, \cite{Wi1}.

This line of research produced further links with Huygens' principle of wave
propagation~\cite{Be},  representation of infinite dimensional Lie algebras,
and isomonodromic deformations of differential equations \cite{Ha},
\cite{Du}.
In the latter context, the second author of the current paper
discovered a link between the theory of the Painlev\'e differential
equations and some families in the $q$-Askey scheme \cite{M}.
Let us briefly explain what this link consists of. The Painlev\'e
differential equations are eight non-linear ODE's whose solutions are
encoded by points in the so-called monodromy manifolds (a different
manifold for each Painlev\'e equation). Each of these monodromy
manifolds carries a natural Poisson structure which quantizes to a
special degeneration of
the AW algebra that regulates
a specific family in the $q$-Askey scheme. Interestingly, dual families
(for example
the continuous dual $q$-Hahn and the big $q$-Jacobi polynomials)
correspond to the same monodromy manifold in the classical limit,
thus suggesting an alternative approach to spot dualities.
Moreover, the limit transitions in the $q$-Askey scheme
correspond to the so called confluence procedure of the
Painlev\'e equations, that can be viewed geometrically
as a procedure to merge holes on a Riemann sphere
by which cusped holes are created \cite{CMR}.
In this picture, dual families in the $q$-Askey scheme
scorrespond to Riemann spheres with the same structure.

This paper studies, for a relatively small but important part of
the $q$-Askey scheme (see Figure \ref{fig:1}),
the duality and its limit behaviour, first
for the symmetric polynomials and the corresponding (degenerate)
AW algebras, and next, starting in  Section \ref{152},
for the non-symmetric polynomials and
the corresponding (degenerate) DAHA's.
These degenerate DAHA's were introduced by the second
author \cite{M}, \cite{M3}. The  ones on the lowest level of Figure \ref{fig:1}
can be recognized as nil-DAHA's \cite[Remark 8.4]{LT}.

\begin{figure}[t]
\setlength{\unitlength}{3mm}
\begin{picture}(14,13)(-4,1)
\linethickness{0.2mm}
\put(16.5,12) {\framebox(9,2.5) {Askey-Wilson}}
\put(18,12) {\vector(-1,-1){2.4}}
\put(24,12) {\vector(1,-1){2.4}}
\put(4,7) {\framebox(15,2.5) {Continuous Dual $q$-Hahn}}
\linethickness{0.4mm}
\multiput(19,8,25)(0.5,0){10}{\line(1,0){0.25}}
\linethickness{0.2mm}
\put(24,7) {\framebox(8,2.5) {Big $q$-Jacobi}}
\put(11.5,7) {\vector(0,-1){2.4}}
\put(28,7) {\vector(0,-1){2.4}}
\put(6,2) {\framebox(11,2.5) {Al-Salam-Chihara}}
\linethickness{0.4mm}
\multiput(17,3.25)(0.5,0){12}{\line(1,0){0.25}}
\linethickness{0.2mm}
\put(23,2) {\framebox(10,2.5) {Little $q$-Jacobi}}
\end{picture}
\caption{Part of $q$-Askey scheme treated in this paper.
Dashed lines show duality.}
\label{fig:1}
\end{figure}
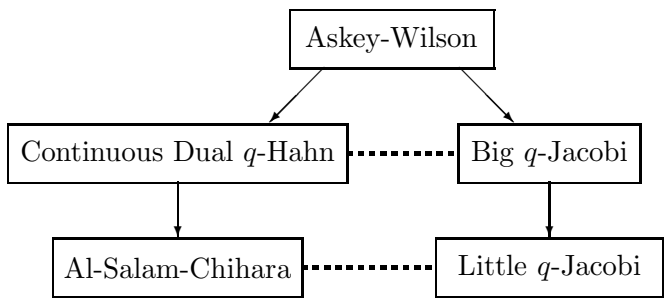

The non-symmetric versions of the continuous dual $q$-Hahn polynomials
and the Al-Salam-Chihara polynomials were earlier studied by
the second author \cite{M3}. With regard to the non-symmetric versions
of the big and little $q$-Jacobi polynomials there is the problem that
one has to pass from Laurent polynomials to ordinary polynomials.
In a paper \cite{KB} by the first author and Bouzeffour this was
circumvented for the limit from Askey-Wilson directly to little $q$-Jacobi
by rewriting the non-symmetric AW polynomials as 2-vector-valued
ordinary polynomials and then taking the limit. As shown in
\S\ref{153}, this works also for the limit from AW to big $q$-Jacobi.
As for non-symmetric little $q$-Jacobi, there turn out to be two
versions, depending on how the limit from big to
little $q$-Jacobi is taken. One of these versions
is dual to Al-Salam-Chihara,
but the other is dual to the Askey-Wilson $q$-Bessel functions
\cite[(2.12)]{KS}, which are no longer polynomials but transcendental
functions. This should not be seen as a serious obstacle.
There are many other examples of non-polynomial limit cases
of polynomials in the ($q$-)Askey scheme, the best known probably
being Bessel functions as limit cases of Jacobi polynomials.
\par
This paper is organized as follows.
Sections 2-4 deal with symmetric polynomials, their duals and the
corresponding (degenerate) Zhedanov algebras.
This is done for the AW polynomials in
Section 2, for contnuous dual $q$-Hahn and big $q$-Jacobi in Section 3,
and for Al-Salam-Chihara and little $q$-Jacobi in Section 4.
Next Sections 5--7 treat non-symmetric polynomials, their duals, and
the corresponding (degenerate) DAHA'a. The non-symmetric
AW case is in Section 5, its 2D vector-valued realization in Section 6,
and the degenerate cases in Section 7. Finally Section 8 gives a summary
of other related work and offers perspectives for further research.
\paragraph{\bf Notation}
In this paper we denote the variable of a Laurent polynomial by $z$ and the
one of a standard polynomial by $x$. To emphasize the type of polynomials
under consideration, we also use 
square brackets for Laurent polynomials and round brackets for
ordinary polynomials. Correspondingly, when dealing with DAHA and its
degenerations, we will denote the generators in ``standard
presentation" by $T_0,T_1,Z^{\pm 1}$ when $Z$ is invertible,
$T_0,T_1,X, X'$ when $X$ is not invertible.

For $q$-hypergeometric series we use notation as in
\cite[Section 1.2]{GR}, but we will usually relax the conditions on $q$.
\section{Duality for the Askey-Wilson polynomials and the Zhedanov algebra}
\subsection{Definition of Askey-Wilson polynomials and eigenvalue
equations}
In this paper we will use the following standardization and
notation for \emph{Askey-Wilson polynomials} (in short
\emph{AW polynomials})
\label{133}
\begin{equation}
R_n[z]=R_n[z;a,b,c,d\,|\,q]:=
\qhyp43{q^{-n},q^{n-1}abcd,az,az^{-1}}{ab,ac,ad}{q,q},
\label{62}
\end{equation}
and we will work in the following assumptions:
\begin{equation}
\begin{split}
&q\ne0,\qquad\qquad\;\;\, q^m\ne1\qquad\quad(m=1,2,\ldots);\\
&a,b,c,d\ne0,\qquad
abcd\ne q^{-m}\quad(m=0,1,2,\ldots).
\end{split}
\label{eq:6}
\end{equation}
The polynomials \eqref{62} are related to the AW polynomials
$p_n(x;a,b,c,d\,|\,q)$ in usual notation \cite[(14.1.1)]{KLS} by
\begin{equation}
p_n\big(\thalf(z+z^{-1});a,b,c,d\,|\,q\big)=a^{-n}(ab,ac,ad;q)_n\,
R_n[z;a,b,c,d\,|\,q].
\label{130}
\end{equation}
While $p_n(x;a,b,c,d\,|\,q)$ is symmetric in its four parameters
$a,b,c,d$, $R_n[z;a,b,c,d\,|\,q]$ is only symmetric in $b,c,d$.
But the larger symmetry involving $a$ is lost anyhow with the duality
to be discussed later, see \eqref{65}.
\par
The polynomials $R_n[z]$ are eigenfunctions of the operator
$L_z$ acting on the space
of symmetric Laurent polynomials $f[z]=f[z^{-1}]$:
\begin{align}
(Lf)[z]&=
L_z\big(f[z]\big):=\big(1+q^{-1}abcd\big)f[z]+\frac{(1-az)(1-bz)(1-cz)(1-dz)}{(1-z^2)(1-qz^2)}\,
\bigl(f[qz]-f[z]\bigr)\nonu\\
&\qquad+\frac{(a-z)(b-z)(c-z)(d-z)}{(1-z^2)(q-z^2)}\,
\bigl(f[q^{-1}z]-f[z]\bigr).
\label{eq:8}
\end{align}
The eigenvalue equation is
\begin{equation}
L_z\big(R_n[z]\big)=\la_n\,R_n[z],\qquad
\la_n:=q^{-n}+abcd q^{n-1}.
\label{eq:22}
\end{equation}
Under condition \eqref{eq:6} all eigenvalues in \eqref{eq:22} are distinct.
\par
The three-term recurrence relation \cite[(14.1.4)]{KLS} for
the AW polynomials can be interpreted as
an eigenvalue equation if we consider $R_n[z]$ for fixed $z$
in its dependence on~$n$. Then $R_n[z]$ is an eigenfunction of the
operator $M_n$, acting on functions $g(n)$ of $n$ ($n=0,1,2,\ldots$),
defined by
\begin{align}
M_n\big(g(n)\big)&:=A_n g(n+1)+(a+a^{-1}-A_n-C_n)g(n)+C_n g(n-1),
\label{12}\\
A_n&:=\frac{(1-abq^n)(1-acq^n)(1-adq^n)(1-abcdq^{n-1})}
{a(1-abcdq^{2n-1})(1-abcdq^{2n})}\,,\nonu\\
C_n&:=\frac{a(1-q^n)(1-bcq^{n-1})(1-bdq^{n-1})(1-cdq^{n-1})}
{(1-abcdq^{2n-2})(1-abcdq^{2n-1})}\,,\quad C_0=0.\nonu
\end{align}
The eigenvalue equation is
\begin{equation}
M_n\big(R_n[z]\big)=(z+z^{-1})\,R_n[z].
\label{33}
\end{equation}

Under stricter conditions than \eqref{eq:6}, namely $0<q<1$ and
$|a|,|b|,|c|,|d|\le1$ such that pairwise products of $a, b, c, d$ are not equal
to 1 and such that non-real parameters occur in complex conjugate pairs,
the AW polynomials are orthogonal with respect to a non-negative weight
function on $x=\thalf(z+z^{-1})\in[-1,1]$. For convenience we give
this orthogonality in the variable $z$ on the unit circle, where the integrand
is invariant under $z\to z^{-1}$:
\begin{equation}
\frac{(q,ab,ac,ad,bc,bd,cd;q)_\iy}{4\pi (abcd;q)_\iy}
\int_{|z|=1}R_m[z]\,R_n[z]
\left|\frac{(z^2;q)_\iy}{(az,bz,cz,dz;q)_\iy}\right|^2 \frac{dz}{iz}
=h_n\,\de_{m,n},
\label{131}
\end{equation}
where $h_0=1$ and
where the explicit expression for $h_n$ (omitted here)
can be obtained from
\cite[(14.1.2)]{KLS} together with \eqref{130}.
\subsection{Zhedanov algebra}
{The {\em Zhedanov algebra} or {\em AW algebra}}
$\AW(3)$ (see \cite{Zhe}) is the algebra with two generators
$K_0$, $K_1$ and with two relations
\begin{equation}
\begin{split}
(q+q^{-1})K_1K_0K_1-K_1^2K_0-K_0K_1^2&=B\,K_1+ C_0\,K_0+D_0,\\
(q+q^{-1})K_0K_1K_0-K_0^2K_1-K_1K_0^2&=B\,K_0+C_1\,K_1+D_1.
\end{split}
\label{1}
\end{equation}
Here the structure constants
$B$, $C_0$, $C_1$, $D_0$, $D_1$ are fixed complex constants.
\begin{remark}
The relations for $\AW(3)$ were originally given in \cite{Zhe} in
terms of three generators
{(which explains the notation $\AW(3)$)}: $K_0$, $K_1$, and in addition $K_2$
which is given in terms of $K_1$ and $K_2$ by
the \emph{$q$-commutator}
\begin{equation*}
K_2:=[K_0,K_1]_q:=q^\half K_0 K_1-q^{-\half} K_1 K_0.
\end{equation*}
This presentation is in particular
suitable for computations
in computer algebra, since the three relations can be written
in PBW form. In this paper we prefer the two generators version
because it makes the duality we plan to discuss more transparent.
\end{remark}
There is a {\em Casimir element} $Q$ commuting with $K_0$, $K_1$:
\begin{multline}
Q:=K_1K_0K_1K_0-(q^2+1+q^{-2})K_0K_1K_0K_1+
(q+q^{-1})K_0^2K_1^2\\
+(q+q^{-1})(C_0K_0^2+C_1K_1^2)
+B\bigl((q+1+q^{-1})K_0K_1+K_1K_0\bigr)\\
+(q+1+q^{-1})(D_0K_0+D_1K_1).
\label{3}
\end{multline}
\begin{remark}
\label{2}
As observed in \cite[Remark 2.3]{K1}, for the five structure constants
$B$, $C_0$, $C_1$, $D_0$, $D_1$ in the relations \eqref{1}
two degrees of freedom are caused
by scale transformations $K_0\to c_0K_0$ and $K_1\to c_1K_1$
of the generators. These induce the following transformations
on the structure constants:
\begin{equation*}
B\to c_0c_1B,\quad
C_0\to c_1^2C_0,\quad
C_1\to c_0^2C_1,\quad
D_0\to c_0c_1^2D_0,\quad
D_1\to c_0^2c_1D_1.
\end{equation*}
These also result into a transformation $Q\to c_0^2 c_1^2 Q$
of the Casimir element \eqref{3}. So there are essentially only
three degrees of freedom for the structure constants
(and one more freedom to fix the value of $Q$ in the basic
representation, seeRemark \ref{94}).
A nice way of presenting this symmetrically was emphasized by
Terwilliger \cite[(1.1)]{Terw}. In slightly different notation this
is done as follows. Put
\begin{equation}
\begin{split}
A_0&:=(q-q^{-1}) C_1^{-\half} K_0,\\
A_1&:=(q-q^{-1}) C_0^{-\half} K_1,\\
A_2&:=(q-q^{-1}) (C_0C_1)^{-\half}
\big(-[K_0,K_1]_q+(q^\half-q^{-\half})^{-1}B\big).
\end{split}
\label{98}
\end{equation}
Then we can equivalently describe $\AW(3)$ as the algebra generated by
$A_0,A_1,A_2$ with relations
\begin{equation}
\begin{split}
(q-q^{-1})^{-1}[A_1,A_2]_q+A_0&=\al_0,\\
(q-q^{-1})^{-1}[A_2,A_0]_q+A_1&=\al_1,\\
(q-q^{-1})^{-1}[A_0,A_1]_q+A_2&=\al_2,
\end{split}
\label{92}
\end{equation}
where $\al_0,\al_1,\al_2$ are structure constants which can be expressed
in terms of $B,C_0,C_1,D_0,D_1$ by
\begin{equation}
\al_0=-\,\frac{(q-q^{-1})D_0}{C_0C_1^\half},\quad
\al_1=-\,\frac{(q-q^{-1})D_1}{C_0^\half C_1},\quad
\al_2=\frac{(q^\half+q^{-\half})B}{(C_0C_1)^\half}.
\label{96}
\end{equation}
Terwilliger \cite{Terw} considers $\al_0,\al_1,\al_2$
as central elements. He calls the resulting algebra the
\emph{universal Askey-Wilson algebra}.
He also identifies a Casimir element $\om$,
{which is closely related to $Q$ in \eqref{3}}:
\begin{multline}
\om:=q^\half A_0A_1A_2+q A_0^2+q^{-1} A_1^2+q A_2^2
-(1+q)\al_0 A_0\\
-(1+q^{-1})\al_1 A_1-(1+q)\al_2 A_2
=(q-q^{-1})^2(C_0C_1)^{-1} Q-\al_2.
\label{95}
\end{multline}
Then he proves in \cite[Corollary 8.3]{Terw} that the four central
elements $\al_0,\al_1,\al_2,\om$ generate the center of the
universal Askey-Wilson algebra.
Therefore, in our presentation, $Q$ generates the center of $\AW(3)$.
\par
To go back from relations \eqref{92} to \eqref{1} we need two
arbitrary rescaling constants $c_0,c_1\ne0$ and then put:
\begin{align*}
&K_0=c_0^{-1}A_0,\quad K_1=c_1^{-1}A_1,\quad
C_0=(q-q^{-1})^2 c_1^{-1},\quad C_1=(q-q^{-1})^2 c_0^{-1},\\
&B=\frac{(q-q^{-1})^2\al_2}{(q^\half+q^{-\half})c_0c_1}\,,\quad
D_0=-\,\frac{(q-q^{-1})^2\al_0}{c_0c_1^2},\quad
D_1=-\,\frac{(q-q^{-1})^2\al_1}{c_0^2c_1}\,.
\end{align*}
Then $Q=(q-q^{-1})^{-2} (c_0c_1)^2(\om+\al_2)$.
\end{remark}
\begin{remark}
\label{91}
In connection with the relations \eqref{1} defining $\AW(3)$
let $\lan K_1,K_2\ran$ denote the free algebra
generated by $K_1$ and $K_2$.
Note that the algebra isomorphism
$\tau\colon \lan K_1,K_2\ran\to\lan K_1,K_2\ran^{\rm op}$
which reverses the order of the factors in the terms of the elements of
$\lan K_1,K_2\ran$,  leaves invariant the ideal
generated by the relations
\eqref{1} (each of the two relations separately
is even left invariant). So $\tau$ induces an algebra
isomorphism $\tau\colon \AW(3)\to \AW(3)^{\rm op}$.
It can be shown that the Casimir element $Q$, given by \eqref{3},
is invariant under $\tau$. However, in the set-up with generators
$A_0,A_1,A_2$ and relations \eqref{92} there is no invariance
of the relations after
reversion of the order of the factors.
\end{remark}
Let $e_1,e_2,e_3,e_4$ be the elementary symmetric polynomials in
$a,b,c,d$:
\begin{equation}
\begin{split}
&e_1:=a+b+c+d,\qquad
e_2:=ab+ac+bc+ad+bd+cd,\\
&\qquad\quad
e_3:=abc+abd+acd+bcd,\qquad
e_4:=abcd.
\end{split}
\label{4}
\end{equation}
Then express the structure constants in \eqref{1}
in terms of $a,b,c,d$ by means of \eqref{4}:
\begin{equation}
\begin{split}
& B :=(1-q^{-1})^2(e_3+qe_1),\\
&C_0 :=(q-q^{-1})^2,\quad
C_1 :=q^{-1}(q-q^{-1})^2 e_4,\\
&D_0 :=-q^{-3}(1-q)^2(1+q)(e_4+qe_2+q^2),\\
&D_1 :=-q^{-3}(1-q)^2(1+q)(e_1e_4+qe_3).
\end{split}
\label{16}
\end{equation}
Note that, for given $C_0=(q-q^{-1})^2$ and
for given values of $B,C_1,D_0,D_1$ we can solve (2.11) as a system
of equations in $e_1,e_2,e_3,e_4$. This system is uniquely solvable.
Next, $e_1,e_2,e_3,e_4$ determine $a,b,c,d$ up to permutations.
\par
There is a representation (the {\em basic representation} or
{\em polynomial representation}) of the algebra
$\AW(3)$ with structure constants \eqref{16}
on the space of symmetric Laurent polynomials
as follows:
\begin{equation}
(K_0f)[z]:=L_z\big(f[z]\big),\qquad
(K_1f)[z]:=(Z+Z^{-1})(f)[z]=(z+z^{-1})f[z],
\label{700}
\end{equation}
where $L_z$
is the operator \eqref{eq:8} having the
AW polynomials
as eigenfunctions and $Z^{\pm 1}$ is the operator of multiplication by
$z^{\pm 1}$.
The Casimir element $Q$ becomes constant in this representation:
\begin{equation}
(Qf)(z)=Q_0\,f(z),
\label{9}
\end{equation}
where
\begin{multline}
Q_0:=q^{-4}(1-q)^2\Bigl(q^4(e_4-e_2)+q^3(e_1^2-e_1e_3-2e_2)\\
-q^2(e_2e_4+2e_4+e_2)
+q(e_3^2-2e_2e_4-e_1e_3)+e_4(1-e_2)\Bigr).
\label{100}
\end{multline}
\begin{remark}
\label{94}
The basic representation \eqref{700} gives rise to a one-parameter
family of representations of $\AW(3)$ by using a scale transformation
$K_0\to \la K_0$, $K_1\to K_1$ in \eqref{1}.
Compare with the beginning of
Remark \ref{2}: we now take $c_0=\la$, $c_1=1$. Now we have
to solve $e_1,e_2,e_3,e_4$
from the system of equations
 \begin{equation*}
\begin{split}
&\quad \la B=(1-q^{-1})^2(e_3+qe_1),\qquad
\la^2 C_1=q^{-1}(q-q^{-1})^2 e_4,\\
&\la D_0 =-\frac{(1-q)^2(1+q)}{q^3}(e_4+qe_2+q^2),\\ 
&\la^2 D_1 =-\frac{(1-q)^2(1+q)}{q^3}(e_1e_4+qe_3),
\end{split}
\end{equation*}
and we get $a,b,c,d$ (depending on $\la$) from $e_1,e_2,e_3,e_4$.
Then, for each value of $\la$ we have a representation
\begin{equation*}
(K_0f)[z]:=\la^{-1} L_z\big(f[z]\big),\qquad
(K_1f)[z]:=(Z+Z^{-1})(f)[z].
\end{equation*}
Here $L_z$ depends on $a,b,c,d$, and hence on $\la$.
Then $Q$ takes the value $\la^{-2}Q_0$ with $Q_0$ given by \eqref{100},
where $e_1,e_2,e_3,e_4$ depend on $\la$.

If, conversely, we do not pick $\la$ but fix $Q_0$ then a very
complicated system of five equations in $e_1,e_2,e_3,e_4,\la$ has
to be solved.

For the relations \eqref{92} a one-parameter family
of representations can be obtained by first passing to the relations
\eqref{1} as we specified in Remark \ref{2}, obtaining the
representations there, and rewriting everything again in terms of
the relations \eqref{92}.
\end{remark}
\begin{definition}
{The \emph{centerfree Zhedanov} or \emph{Askey-Wilson algebra}
$\AW(3,Q_0)$} is the algebra generated
by $K_0,K_1$ with three relations, namely the two relations
\eqref{1}, where the structure constants are expressed in terms
of $a,b,c,d,q$ by \eqref{16} and \eqref{4}, and the relation
\begin{equation}
Q=Q_0,
\label{93}
\end{equation}
where $Q$ and $Q_0$ are given by \eqref{3} and \eqref{100}.
\end{definition}
In order to emphasize the dependence on the structure constants and
the choice of generators, we will also use notation
$\AW(3,Q_0)=\AW_{a,b,c,d;q}(3,Q_0)
=\AW_{a,b,c,d;q}(3,Q_0;K_0,K_1)$.
By Remark \ref{94}, for $q$ fixed, $\AW_{a,b,c,d;q}(3,Q_0)$ is in
bijective correspondence with $a,b,c,d$ up to permutations.
By what we observed at the end of Remark \ref{2} the algebra
$\AW(3,Q_0)$ has center $\{0\}$.
It was proved in \cite[Theorem~2.2]{K} that \eqref{700} generates a
faithful representation of $\AW_{a,b,c,d;q}(3,Q_0)$.
By Remark \ref{91} the map
$\tau\colon \lan K_1,K_2\ran\to\lan K_1,K_2\ran^{\rm op}$
induces an algebra isomorphism
$\tau\colon \AW(3,Q_0)\to \AW(3,Q_0)^{\rm op}$.
\par
A representation of $\AW_{a,b,c,d;q}(3,Q_0)$ which is essentially
equivalent to the
representation generated by
\eqref{700} can be realized on the space of functions
$g(n)$ ($n=0,1,2,\ldots$) as follows:
\begin{equation}
(K_0g)(n):=\La_n(g(n)):=\la_n\,g(n),\qquad
(K_1g)(n):=M_n(g(n)).
\label{19}
\end{equation}
This follows because the AW polynomials are the
{\it overlap coefficients}\/ connecting the two representations:
\begin{equation*}
L_z(R_n[z])=\La_n(R_n[z]),\qquad
(Z+Z^{-1})(R_n[z])=M_n(R_n[z])
\end{equation*}
in the following sense: 
the AW polynomials form a complete system of orthogonal polynomials
with respect to a suitable orthogonality measure $\mu$.
Then the Fourier-Askey-Wilson transform $f\to \wh f$,
\begin{equation}
\wh f(n):=\int f[z]\,R_n[z]\,d\mu(z)
\label{39}
\end{equation}
intertwines between the two representations. In fact:
\begin{equation}
\begin{split}
\Lambda_n\big(\,\wh f(n)\big)=\lambda_n\int f[z]\,R_n[z]\,d\mu(z)=
\int f[z]\,(LR_n)[z]\,d\mu(z)\\
=\int (Lf)[z]\,R_n[z]\,d\mu(z)=\wh{Lf}(n),
\end{split}
\label{41}
\end{equation}
where in the last step we have used the fact that  $L$
is a self-adjoint operator on the real Hilbert space $L^2(d\mu)$.
Similarly,
\begin{equation*}
M_n\big(\,\wh f(n)\big)=\int f[z]\,M_n(R_n[z])\,d\mu(z)=\int
\left(z+z^{-1}\right) f[z]\,R_n[z]\,d\mu(z)
=\big((Z+Z^{-1})(f)\big)\hat{\vphantom{1}}\,(n).
\end{equation*}
More generally, if $p(K_0,K_1)\in\lan K_0,K_1\ran$ then
\begin{equation*}
p(\La_n,M_n)\big(\,\wh f(n)\big)=
\big(p(L,Z+Z^{-1})(f)\big)\hat{\vphantom{1}}\,(n).
\end{equation*}
Hence, $\La,M$ satisfy the same relations \eqref{1}, \eqref{93}
as $L,Z+Z^{-1}$. Thus \eqref{19} generates a representation of
$\AW_{a,b,c,d;q}(3,Q_0)$.
This was already observed in \cite{Zhe} and \cite{GLZ},
more concretely for the $q$-Racah case, where the representations
are finite-dimensional.
\par
By faithfulness we have
\begin{equation*}
\AW_{a,b,c,d;q}(3,Q_0;K_0,K_1)\simeq
\AW_{a,b,c,d;q}(3,Q_0;L,Z+Z^{-1})\simeq
\AW_{a,b,c,d;q}(3,Q_0;\La,M).
\end{equation*}
\subsection{Duality for AW polynomials}
Define {\em dual parameters}
$\td a,\td b,\td c,\td d$ in terms of $a,b,c,d$ by
\begin{equation}
\td a=(q^{-1}abcd)^\half,\quad \td b=ab/\td a,\quad \td c=ac/\td a,\quad
\td d=ad/\td a.
\label{64}
\end{equation}
Jumping from one branch to the other branch in the square root in
the formula for $\td a$ implies that $\td a,\td b,\td c,\td d$ move to
$-\td a,-\td b,-\td c,-\td d$. This corresponds to the following trivial 
symmetry that follows immediately from \eqref{62}:
\begin{equation}\label{63}
R_n[z;a,b,c,d\,|\, q]=R_n[-z;-a,-b,-c,-d\,|\, q].
\end{equation}
Repetition of the parameter transformation recovers
the original parameters up to a possible common multiplication of
$a,b,c,d$ by $-1$, while the branch choice for $\td a$ is irrelevant:
\begin{equation}
a=\big(q^{-1}\td a\td b\td c\td d\,\big)^\half,\quad b=\td a\td b/a,\quad
c=\td a\td c/a,\quad d=\td a\td d/a.
\label{65}
\end{equation}
From \eqref{62} we have the duality relation
\begin{equation}
R_n\big[a^{-1}q^{-m};a,b,c,d\,|\,q\big]=
R_m\big[\td a^{-1} q^{-n};\td a,\td b,\td c,\td d\,|\,q\big]\qquad
(m,n\in\ZZ_{\ge0}).
\label{66}
\end{equation}
By \eqref{63} the two sides of \eqref{66} are invariant under
common multiplication by $-1$ of $a,b,c,d$, respectively
$\td a,\td b,\td c,\td d$.
\par
There is a duality corresponding to \eqref{66} for the operators
$L_z$ and $M_n$ defined by \eqref{eq:8} and \eqref{12}, respectively:
\begin{equation}
L_z\big(f[z]\big)\Big|_{z=a^{-1}q^{-m}}=
\td a\,\td M_m\big(f\big[a^{-1}q^{-m}\big]\big),
\label{18}
\end{equation}
where $\td M_m$ is the difference operator $M_m$ with
respect to dual parameters.
For $f[z]:=R_n[z]$ both sides of \eqref{18}
yield $(q^{-n}+abcdq^{n-1}) R_n[a^{-1}q^{-m}]$.
Similarly to \eqref{18} there is a duality between the multiplication
operators $Z+Z^{-1}$ given by \eqref{700} and $\La_n$
given by \eqref{19}:
\begin{equation}
(Z+Z^{-1})(f[z])\big|_{z=a^{-1}q^{-m}}=
a^{-1}\,\td\La_m\big(f(a^{-1}q^{-m})\big),
\label{23}
\end{equation}
where $\td\La_m$ is the multiplication operator $\La_m$ with
respect to dual parameters.
\par
Formulas \eqref{19} and \eqref{23} are instances of operators
$\check A$
acting on functions on $\ZZ_{\ge0}$ which are induced by restriction
of operators $A$ acting on functions $f[z]=f[z^{-1}]$ depending
on $z\in\mathbb{C}\backslash\{0\}$. Suppose that such an operator
$A$ has the property that $(Af)[a^{-1}q^{-m}]=0$ for all
$m\in\ZZ_{\ge0}$
if $f[a^{-1}q^{-m}]=0$ for all $m\in\ZZ_{\ge0}$. Then put
\[
(\check A g)(m):=(Af)[a^{-1}q^{-m}]\quad{\rm if}\quad
g(m)=f[a^{-1}q^{-m}].
\]
Clearly $(AB)\check{\vphantom{1}}=\check A \check B$.
By \eqref{19} and \eqref{23}, if $A=L$ then $\check A=\td a\td M$, and
if $A=Z+Z^{-1}$ then $\check A=a^{-1}\td\La$.
\par
Corresponding to the trivial symmetry
\begin{equation}
R_n[z;a,b,c,d\,|\, q]=R_n\big[z;a^{-1},b^{-1},c^{-1},d^{-1}\,
\big|\,q^{-1}\big],
\label{11}
\end{equation}
we see that, by  \eqref{eq:8}, $L_z$
becomes $\frac q{abcd}\,L_z$ if
$a,b,c,d,q\to a^{-1},b^{-1},c^{-1},d^{-1},q^{-1}$.
\begin{remark}
\label{rmk:14}
By \cite[\S 5.7,\S 8.5]{NS} our dual parameters \eqref{64} match with
the dual parameters in~\cite{NS}: just interchange $k_0$ and $u_1$
in \cite[\S5.7]{NS}. 
\end{remark}
\subsection{{Duality for $\AW(3,Q_0)$}}\label{suse:dZ}
There are several symmetries of $\AW_{a,b,c,d;q}(3,Q_0;K_0,K_1)$.
We already observed that it is invariant under permutations of
$a,b,c,d$.
\par
There is an isomorphism
\cite[\S2]{Zhe}, \cite[(2.11)]{K1}
\footnote{In \cite[(2.11)]{K1} the fourth argument of $\AW$ on the
right-hand side should be $(q^{-1}abcd)^\half$.}
(both an algebra and an anti-algebra isomorphism):
\begin{equation}
{\AW_{a,b,c,d;q}\Big(3,Q_0;K_0,K_1\Big)\simeq
\AW_{\td a,\td b,\td c,\td d;q}\Big(3,\td Q_0;a K_1,\td a^{-1}K_0\Big)},
\label{70}
\end{equation}
where $\td a, \td b,\td c,\td d$ are given in \eqref{64},
{and $\td Q_0$ denotes $Q_0$ in terms of the dual parameters}.
Indeed, if  $\td B,\td C_0,\ldots$ denote $B,C_0,\ldots$
in terms of the dual parameters then
\begin{equation}
\begin{split}
&\td B=a\td a^{-2}\,B,\quad
\td C_0=\td a^{-2}\,C_1,\quad
\td C_1=a^2\,C_0,\\
&\td D_0=a\td a^{-2}\,D_1,\quad
\td D_1=a^2\td a^{-1}\,D_0,\quad
\td Q_0=a^2\td a^{-2}\,Q_0.
\end{split}
\label{97}
\end{equation}
{Hence relations \eqref{1} with dual parameters and with
$K_0,K_1$ replaced by $a K_1,\td a^{-1}K_0$ are equivalent to the
original relations \eqref{1}.}
Furthermore replacement of $K_0,K_1,a,b,c,d$ in the right-hand side
of \eqref{3} by $a K_1,\td a^{-1}K_0,\td a,\td b,\td c,\td d$,
respectively, yields the old expression multiplied by $a^2\td a^{-2}$,
{and, by \eqref{97}, the same is true if we replace $a,b,c,d$ by
$\td a,\td b,\td c,\td d$ in the right-hand side of \eqref{100}.}
So the algebras on the left and right of \eqref{70} satisfy
equivalent relations.
\begin{remark}
\label{99}
Let us consider the effect of the duality \eqref{70} on the
representation \eqref{700}. This being a representation means
that the relations \eqref{1}, {\eqref{3}} hold for $K_0=L$,
$K_1=Z+Z^{-1}$. By \eqref{70} these relations with $a,b,c,d$
replaced by $\td a,\td b,\td c,\td d$ hold for $K_0=a(Z+Z^{-1})$,
$K_1=\td a^{-1}L$. Then, by \eqref{18} and \eqref{23}, the same
relations also hold for $K_0=\td\La$, $K_1=\td M$. Thus we have arrived
via the duality isomorphism \eqref{70} at the representation
\eqref{19} with $a,b,c,d$ replaced by $\td a,\td b,\td c,\td d$.
\end{remark}
\begin{remark}
The duality \eqref{70} takes a particularly simple and elegant form
for the algebra generated by $A_0,A_1,A_2$ with relations \eqref{92}
together with $\om=\om_0$, where $\om$ is
given by \eqref{95} and $\om_0$ is a constant.
By \eqref{98}, \eqref{96} and \eqref{97} we see that the duality
for that algebra amounts to an anti-isomorphism which interchanges
$A_0$ and $A_1$ and keeps $A_2$ fixed, while $\al_0$ and $\al_1$ are
interchanged and $\al_2$ and $\om_0$ are kept fixed.
It follows from \cite[Lemma 6.1]{Terw} that $\om$ then does not change.
Since the parameters $a,b,c,d$ are no longer involved in
this formulation of the duality, the symmetry breaking in
\eqref{70} seems to be absent now. The price to be paid for this
is that there is less immediate contact with the AW
polynomials.
It is also clear that the duality in this setting is part of an
$S_3$ symmetry acting simultaneously on $A_0,A_1,A_2$ and
$\al_0,\al_1,\al_2$.
{The reparametrization of the Askey-Wilson parameters in
Huang \cite[\S3.2]{Huang} seems to behave nicely under this action
of $S_3$.
It is not clear what would be the effect
on the basic representation by the action of
the full $S_3$ symmetry group.}
\end{remark}
There is also an algebra isomorphism
\begin{equation}
\AW_{a,b,c,d;q}\Big(3,Q_0;K_0,K_1\Big)\simeq
\AW_{a^{-1},b^{-1},c^{-1},d^{-1};q^{-1}}\left(3,Q_0;
\frac q{abcd}\,K_0,K_1\right).
\label{156}
\end{equation}
\section{Duality for continuous dual \boldmath$q$-Hahn and big
\boldmath$q$-Jacobi polynomials}\label{se:dcqH-BqJ}
\subsection{Limits to Continuous Dual $q$-Hahn and Big $q$-Jacobi}
\paragraph{\bf Limit from AW to
Continuous Dual \boldmath$q$-Hahn}
The continuous dual $q$-Hahn polynomials are the limit case $d\to0$
of the AW polynomials \eqref{62}:
\begin{equation}
R_n[z;a,b,c\,|\,q]:=
\qhyp32{q^{-n},az,az^{-1}}{ab,ac}{q,q}
=\lim_{d\to0}R_n[z;a,b,c,d\,|\,q].
\label{67}
\end{equation}
The polynomials \eqref{67} are related to the continuous dual
$q$-Hahn  polynomials
$p_n(x;a,b,c\,|\,q)$ in usual notation \cite[(14.3.1)]{KLS} by
\begin{equation*}
p_n\big(\thalf(z+z^{-1});a,b,c\,|\,q\big)=a^{-n}(ab,ac;q)_n\,
R_n[z;a,b,c\,|\,q].
\end{equation*}
\par
The corresponding limits of \eqref{eq:8}, \eqref{eq:22},
\eqref{12}, \eqref{33}
for the operators $L$, $M$ and its eigenvalue equations are:
\begin{align}
&(Lf)[z]=L_z\big(f[z]\big)
=\frac{(1-az)(1-bz)(1-cz)}{(1-z^2)(1-qz^2)}\,\bigl(f[qz]-f[z]\bigr)\nonu\\
&\qquad\qquad\qquad\qquad\qquad
-\frac{z(a-z)(b-z)(c-z)}{(1-z^2)(q-z^2)}\,
\bigl(f[q^{-1}z]-f[z]\bigr)+f[z],\label{35}\\
\noalign{\allowbreak}
&L_z\big(R_n[z]\big)=q^{-n}\,R_n[z],\label{36}\\
\noalign{\allowbreak}
&M_n\big(g(n)\big)=a^{-1}(1-abq^n)(1-acq^n)(g(n+1)-g(n))\nonu\\
&\qquad\qquad\qquad
+a(1-q^n)(1-bcq^{n-1})(g(n-1)-g(n))+(a+a^{-1})g(n),\label{37}\\
\noalign{\allowbreak}
&M_n\big(R_n[z]\big)=(z+z^{-1})\,R_n[z].
\label{38}
\end{align}
The obtained $q$-difference equation and recurrence relation agree
with \cite[(14.3.7), (14.3.4)]{KLS}.
\paragraph{\bf Limit from AW to Big \boldmath$q$-Jacobi}
The big $q$-Jacobi polynomials \cite[(14.5.1)]{KLS}
are obtained as a more tricky limit case  \cite[(14.1.18)]{KLS}
of AW polynomials \eqref{62}:
\begin{equation}
\begin{split}
P_n(x;a,b,c;q):=
\qhyp32{q^{-n},q^{n+1}ab,x}{aq,cq}{q,q}
=\lim_{\la\to0}
R_n[\la^{-1}x;\la,qa\la^{-1},qc\la^{-1},bc^{-1}\la\,|\,q].
\label{68}
\end{split}
\end{equation}
\par
The corresponding limits of \eqref{eq:8}, \eqref{eq:22},
\eqref{12}, \eqref{33}
for the operators $L$, $M$ and its eigenvalue equations are:
\begin{align}
&(Lf)(x)=L_x\big(f(x)\big)
=qacx^{-2}(1-x)(1-bc^{-1}x)\bigl(f(qx)-f(x)\bigr)\nonu\\
&\qquad\qquad\qquad
+x^{-2}(qa-x)(qc-x)\,\bigl(f(q^{-1}x)-f(x)\bigr)+(1+qab)f(x),\label{42}\\
&L_x\big(P_n(x)\big)=(q^{-n}+q^{n+1}ab)P_n(x),\label{43}\\
&M_n\big(g(n)\big)=\frac{(1-abq^{n+1})(1-aq^{n+1})(1-cq^{n+1})
(g(n+1)-g(n))}{(1-abq^{2n+1})(1-abq^{2n+2})}
\nonu\\
&\qquad\qquad\quad
-\frac{q^{n+1}ac(1-q^n)(1-bq^n)(1-abc^{-1}q^n)(g(n-1)-g(n))}
{(1-abq^{2n})(1-abq^{2n+1})}
+g(n),\label{44}\\
&M_n\big(P_n(x)\big)=x\,P_n(x).\label{45}
\end{align}
When taking the limit in \eqref{12}, \eqref{33}
we have to substitute
\begin{equation*}
z,a,b,c,d,M_n\;\to\;
\la^{-1}x,\la,qa\la^{-1},qc\la^{-1},bc^{-1}\la,\la M_n
\end{equation*}
The obtained $q$-difference equation and recurrence relation agree
with \cite[(14.5.5), (14.5.3)]{KLS}.
\subsection{Duality between Continuous Dual \boldmath$q$-Hahn and
Big \boldmath$q$-Jacobi}
From the $q$-hypergeometric expressions \eqref{67} and \eqref{68}
we see that
\begin{equation}
R_n\big[a^{-1}q^{-m};a,b,c\,|\,q\big]=
P_m(q^{-n};q^{-1}ab,ab^{-1},q^{-1}ac;q)
\qquad(m,n\in\ZZ_{\ge0}).
\label{69}
\end{equation}
This duality turns out to be a limit case of the Askey-Wilson duality
\eqref{66}.
Indeed, by \eqref{64} we have
\begin{equation}
(a,b,c,d)=(a,b,c,\frac{q\la^2}{a b c})\quad
\Leftrightarrow\quad
(\td a,\td b, \td c,\td d)=
\pm(\la,\frac{ab}{\la},\frac{ac}{\la},\frac{q\la}{bc}).
\label{81}
\end{equation}
So, by \eqref{66},
\begin{equation}\label{eq:d-bqj-cdqh}
R_n\big[a^{-1}q^{-m};a,b,c,\frac{q\la^2}{a b c}\,|\,q\big]
=R_m\big[\la^{-1}q^{-n};\la,\frac{ab}{\la},\frac{ac}{\la},\frac{q\la}{bc}\,|\,q\big].
\end{equation}
Now let $\la\to0$ in the above equality. By the limits
\eqref{67} and \eqref{68}
we obtain the duality \eqref{69}.
\par
Similarly to, and as a limit case of \eqref{18}
there is a duality corresponding to \eqref{69} for the operators
$L_z$ and $M_n$ defined by \eqref{35} and \eqref{44}, respectively:
\begin{equation}
L_z^{a,b,c}\big(f[z]\big)\Big|_{z=a^{-1}q^{-m}}=
\td a\,M_m^{\td a,\td b,\td c}\big(f\big[a^{-1}q^{-m}\big]\big),
\label{57}
\end{equation}
where $L_z^{a,b,c}$ is the operator $L_z$  given by \eqref{35},
$M_n^{a,b,c}$ is the operator $M_n$ given by \eqref{44},
and
\begin{equation}
(\td a,\td b,\td c)=
\big((qab)^\half,(qab^{-1})^\half,(qa^{-1}b^{-1})^\half c\big).
\label{53}
\end{equation}
Note that the map $(a,b,c)\mapsto(\td a,\td b,\td c)$ is inverse to the
map $(a,b,c)\mapsto(q^{-1}ab ,ab^{-1}, q^{-1}ac)$.
\par
There is also a duality corresponding to \eqref{69} for the operators
$L_x$ and $M_n$ defined by \eqref{42} and \eqref{37}, respectively:
\begin{equation}
L_x^{a,b,c}\big(f(x)\big)\Big|_{x=q^{-m}}=
\td a\,M_m^{\td a,\td b,\td c}\big(f(q^{-m})\big),
\label{58}
\end{equation}
where $L_x^{a,b,c}$ is the operator $L_x$  given by \eqref{42},
$M_n^{a,b,c}$ is the operator $M_n$ given by \eqref{37},
and $\td a,\td b, \td c$ are as in \eqref{53}.
\subsection{{Corresponding degenerations of $\AW(3,Q_0)$ and their duality}}

\par
As $d\to 0$ the Zhedanov algebra
$\AW(3,Q_0)=\AW_{a,b,c,d;q}(3,Q_0;K_0,K_1)$
tends to the algebra with two generators
$K_0$, $K_1$ {with relations \eqref{1} and \eqref{9}},
where $C_1=0$ and $B,C_0,D_0,D_1,Q_0$ depend on $a,b,c$
and on $d=0$ as in \eqref{16} and \eqref{100}
{and $e_1,e_2,e_3$ are the elementary symmetric polynomials
in $a,b,c$}:
\begin{equation}
\label{30}
\begin{split}
&B =(1-q^{-1})^2\big(ab+ac+bc+q(a+b+c)\big),\qquad
C_0 :=(q-q^{-1})^2,\\
&D_0 =-\frac{(1-q)^2(1+q)}{q^2}(ab+ac+bc+q),\qquad
D_1 =-\frac{(1-q)^2(1+q)}{q^2}abc,\\
&{Q_0=q^{-3}(1-q)^2\big(-q^3e_2+q^2(e_1^2-e_1e_3-2e_2)-qe_2
+e_3^2-e_1e_3\big)},
\end{split}
\end{equation}
In the expression \eqref{3} for $Q$ also put $C_1=0$ and substitute
\eqref{30}.
We denote the resulting algebra by
\begin{equation}
{\AW_{a,b,c;q}^{\rm CDqH}(3,Q_0)=
\lim_{d\to0}\AW_{a,b,c,d;q}(3,Q_0)},
\label{24}
\end{equation}
{where, if needed, the two generators can be added to the notation.}
The representations \eqref{700} and \eqref{19} also hold for
$\AW_{a,b,c;q}^{\rm CDqH}(3,Q_0)$, but now with $L_z$ and $M_n$
given by \eqref{35} and \eqref{37}, and with $\La_n$ given by
\begin{equation}
\La_n\big(g(n)\big):=q^{-n} g(n).
\label{59}
\end{equation}
The argumentation by \eqref{39} and \eqref{41} for the equivalence
of the two representations also remains valid if we take \eqref{67}
for $R_n$ and if we put $\la_n=q^{-n}$ and if we use \eqref{36} and
\eqref{38}.
\par
Now consider the Zhedanov algebra
$\AW_{\la,qa\la^{-1},qc\la^{-1},bc^{-1}\la;q}\allowbreak(3,Q_0;\allowbreak
K_0,\la^{-1}K_1)$, {rescale $Q=Q_0$ as $\la^2 Q=\la^2 Q_0$
and take the limit as $\la\to  0$, where the new $Q$ and $Q_0$
are the limits of $\la^2 Q$ and $\la^2 Q_0$, respectively}.
This produces the algebra with two generators
$K_0$, $K_1$ and {with relations \eqref{1} and \eqref{9}},
where $C_0=0$ and $B,C_1,D_0,D_1$ are given as follows:
\begin{equation}
\label{eq:z-BqJ}
\begin{split}
&B =(1-q)^2(c+a+ a b+ a c),\quad D_0 = - (1-q)^2(1+q) a c,\\
&C_1=q(q-q^{-1})^2 ab, \quad D_1=  - (1-q)^2(1+q)a (c+b + a b+b c).
 \end{split}
\end{equation}
(We omit the quite lengthy explicit expressions of $Q$ and
$Q_0$.)
Denote the resulting algebra by
\begin{equation}
\AW_{a,b,c;q}^{\rm BqJ}(3,Q_0;K_0,K_1)=
\lim_{\la\to0}
\AW_{\la,qa\la^{-1},qc\la^{-1},bc^{-1}\la;q}(3,\la^2 Q_0; K_0, \la^{-1}K_1).
\label{27}
\end{equation}
For $\AW_{a,b,c;q}^{\rm BqJ}(3,Q_0;K_0,K_1)$
the representations \eqref{700} and \eqref{19} take
the form
\begin{equation}
(K_0f)(x):=L_x\big(f(x)\big),\qquad
(K_1f)(x)=X(f)(x):=x\,f(x)
\label{55}
\end{equation}
and
\begin{equation}
(K_0g)(n):=\La_n(g(n)):=(q^{-n}+q^{n+1}ab)g(n),\qquad
(K_1g)(n):=M_n(g(n))
\label{56}
\end{equation}
with $L_x$ and $M_n$ defined by \eqref{42} and \eqref{44}.
The argumentation by \eqref{39} and \eqref{41} for the equivalence
of the two representations also remains valid after slight but obvious
adaptations.
\begin{proposition}
There is an isomorphism (both an algebra and an anti-algebra isomorphism)
\begin{equation}
\AW_{a,b,c;q}^{\rm CDqH}(3,Q_0;K_0,K_1)\simeq
\AW_{q^{-1}ab ,ab^{-1}, q^{-1}ac;q}^{\rm BqJ}(3,Q_0;a K_1, K_0).
\label{con-70}
\end{equation}
\end{proposition}
\begin{proof} By substitution of \eqref{81} in \eqref{70} we obtain
\begin{equation}\label{eq:dual-l}
\AW_{a,b,c,\frac{q\lambda^2}{abc},q}(3,Q_0;K_0,K_1)\simeq
\AW_{\lambda,\frac{ab}{\lambda},\frac{ac}{\lambda},\frac{q\lambda}{bc};q}
(3,Q_0;a K_1,\lambda^{-1}K_0).
\end{equation}
{Here, at each of the two sides, $Q_0$ is in terms of the parameters
given at that side.
Then $Q=Q_0$ on the right means $a^2\la^{-2} Q=a^2\la^{-2}Q_0$ or,
after rescaling, $a^2 Q=a^2 Q_0$
in terms of $Q$ and $Q_0$ on the left.}
Now apply the limits \eqref{24} and \eqref{27} to the left and right side
of \eqref{eq:dual-l}, respectively.
\end{proof}
\par
The representation \eqref{700} of
$\AW_{a,b,c;q}^{\rm CDqH}(3,Q_0;K_0,K_1)$
with $L_z$ given by \eqref{35}
corresponds under the isomorphism \eqref{con-70}
to the representation \eqref{56} of
$\AW_{a,b,c;q}^{\rm BqJ}(3,Q_0;K_0,K_1)$.
Similarly, the representation \eqref{19} of
$\AW_{a,b,c;q}^{\rm CDqH}(3,Q_0;K_0,K_1)$ with $M_n$ and $\La_n$
given by \eqref{37} and \eqref{59}, respectively,
corresponds under the isomorphism \eqref{con-70}
to the representation \eqref{55} of
$\AW_{a,b,c;q}^{\rm BqJ}(3,Q_0;K_0,K_1)$.
{Both results follow either by taking suitable limits from the
Askey-Wilson case or by imitating the reasoning in Remark \ref{99},
now using \eqref{57} or \eqref{58}, respectively.}
\section{Duality for Al-Salam-Chihara and little \boldmath$q$-Jacobi
polynomials}
\label{se:ASC-LqJ}
\subsection{Limits to Al-Salam-Chihara and Little $q$-Jacobi}
\paragraph{\bf Limit from Continuous Dual \boldmath$q$-Hahn to
Al-Salam-Chihara}
The Al-Salam-Chihara polynomials are the limit case $c\to0$
of the continuous dual $q$-Hahn polynomials~\eqref{67}:
\begin{equation}
R_n[z;a,b\,|\,q]:=
\qhyp32{q^{-n},az,az^{-1}}{ab,0}{q,q}
=\lim_{c\to0}R_n[z;a,b,c\,|\,q].
\label{eq:ASC}
\end{equation}
The polynomials \eqref{eq:ASC} are related to the 
Al-Salam-Chihara  polynomials $Q_n(x;a,b\,|\,q)$ in usual notation
\cite[(14.8.1)]{KLS} by
\begin{equation*}
Q_n(x;a,b\,|\,q)=a^{-n} (ab;q)_n\,R_n[z;a,b\,|\,q].
\end{equation*}
\par
The corresponding limits of \eqref{35}--\eqref{38}
for the operators $L$, $M$ and its eigenvalue equations are:
\begin{align}
&(Lf)[z]=L_z\big(f[z]\big)
=\frac{(1-az)(1-bz)}{(1-z^2)(1-qz^2)}\,\bigl(f[qz]-f[z]\bigr)\nonu\\
&\qquad\qquad\qquad\qquad\qquad
+\frac{z^2(a-z)(b-z)}{(1-z^2)(q-z^2)}\,
\bigl(f[q^{-1}z]-f[z]\bigr)+f[z],\label{82}\\
\noalign{\allowbreak}
&L_z\big(R_n[z]\big)=q^{-n}\,R_n[z],\label{83}\\
&M_n\big(g(n)\big)=a^{-1}(1-abq^n)(g(n+1)-g(n))
\nonu\\
&\qquad\qquad\qquad\qquad
+a(1-q^n)(g(n-1)-g(n))+(a+a^{-1})g(n),\label{84}\\
\noalign{\allowbreak}
&M_n\big(R_n[z]\big)=(z+z^{-1})\,R_n[z].\label{85}
\end{align}
The obtained $q$-difference equation and recurrence relation agree
with \cite[(14.8.7), (14.8.4)]{KLS}.
\paragraph{\bf Limit from Continuous Dual \boldmath$q$-Hahn to
Askey-Wilson $q$-Bessel}
The Askey-Wilson $q$-Bessel functions are defined
as follows (see \cite[(2.12)]{KS}):
\begin{equation*}
J_\ga[z;a,b\,|\,q]:=\qhyp21{az,az^{-1}}{ab}{-q\ga a^{-1}}.
\end{equation*}
By \eqref{67} they can be obtained as a limit case of
continuous dual $q$-Hahn polynomials:
\begin{equation}
J_{q^n\ga}[z;a,b\,|\,q]=\lim_{N\to\iy}
R_{N-n}\big[z;a,b,-q^{-N}\ga^{-1}\,|\,q\big]\qquad(n\in\ZZ).
\label{148}
\end{equation}
Then the orthogonality relations \cite[(2.14)]{KS} for the functions
$J_{q^n\ga}$ follow, under suitable constraints on the parameters,
from the orthogonality relations \cite[(14.3.2)]{KLS} for
continuous dual $q$-Hahn polynomials.
If we rescale $L^{a,b,c;q}$ in \eqref{37} and next take a limit
for $N\to\iy$ of \eqref{38} in the form
\begin{multline*}
q^N L_z^{a,b,-q^{-N}\ga^{-1}}\
R_{N-n}\big[z^{-1};a,b,-q^{_N}\ga^{-1}\,|\,q\big]\\
=q^n R_{N-n}\big[z^{-1};a,b,-q^{_N}\ga^{-1}\,|\,q\big]\quad
(-\iy<n\le N)
\end{multline*}
then $L:=\lim_{N\to\iy}q^N L_z^{a,b,-q^{-N}\ga^{-1}}$ exists and
\begin{equation}
L_z J_{q^n\ga}[z;a,b\,|\,q] =q^n J_{q^n\ga}[z;a,b\,|\,q]\qquad
(n\in\ZZ).
\label{149}
\end{equation}
\paragraph{\bf Limits from Big \boldmath$q$-Jacobi to Little
$q$-Jacobi}
The little $q$-Jacobi polynomials are defined as follows
(see \cite[(14.12.1)]{KLS}):
\begin{equation}
p_n(x;a,b;q):=
\qhyp21{q^{-n},a b q^{n+1}}{a q}{q,qx},
\end{equation}
or, equivalently (by \cite[(3.38)]{K3}):
\begin{equation}\label{eq:LqJequiv}
p_n(x;a,b;q):=
(-qb)^{-n} q^{-\half n(n-1)}\,
\frac{(q b;q)_n}{(q a;q)_n}\,\qhyp32{q^{-n},q^{n+1}ab, q b x}{q b,0}{q,q}.
\end{equation}
\par
Little $q$-Jacobi polynomials appear as limits of big $q$-Jacobi polynomials
in two ways:
\begin{equation}\label{eq:l1}
p_n(x;a,b;q)=\lim_{c\to\iy} P_n(c q x;a,b,c;q)
\end{equation}
and
\begin{equation}\label{eq:l2}
p_n(x;a,b;q)=(-qb)^{-n} q^{-\half n(n-1)}\,
\frac{(q b;q)_n}{(q a;q)_n}\, P_n(qbx;b,a;q),
\end{equation}
where
\begin{equation}
P_n(x;a,b;q)=\lim_{c\to0}P_n(x;a,b,c;q)=
\qhyp32{q^{-n},q^{n+1}ab,x}{aq,0}{q,q}.
\label{28}
\end{equation}
\par
The limits of the operators $L$, $M$ and its eigenvalue equations
in \eqref{42}--\eqref{45} which correspond to the limit \eqref{28}
are as follows:
\begin{align}
&(Lf)(x)=L_x\big(f(x)\big)
=-qabx^{-1}(1-x)\bigl(f(qx)-f(x)\bigr)\nonu\\
&\qquad\qquad\qquad\qquad\qquad
-x^{-1}(qa-x)\,\bigl(f(q^{-1}x)-f(x)\bigr)+(1+qab)f(x),\label{86}\\
&L_x\big(P_n(x)\big)=(q^{-n}+q^{n+1}ab)P_n(x),\label{87}\\
&M_n\big(g(n)\big)=\frac{(1-abq^{n+1})(1-aq^{n+1})}
{(1-abq^{2n+1})(1-abq^{2n+2})}\,(g(n+1)-g(n))\nonu\\
&\qquad\qquad\qquad\qquad
+\frac{q^{2n+1}a^2b(1-q^n)(1-bq^n)}
{(1-abq^{2n})(1-abq^{2n+1})}\,(g(n-1)-g(n))+g(n),\label{88}\\
&M_n\big(P_n(x)\big)=x\,P_n(x).\label{89}
\end{align}
The obtained $q$-difference equation and recurrence relation agree
with \cite[(14.12.5), (14.12.3)]{KLS} if we take into account
\eqref{eq:l2}.
\subsection{Duality between Al-Salam-Chihara and Little $q$-Jacobi}
By taking the limit as $c\to 0$ in \eqref{69}, and by
using \eqref{eq:ASC} and \eqref{28}, we obtain the following duality:
\begin{equation}
R_n\big[a^{-1}q^{-m};a,b\,|\,q\big]
=P_m(q^{-n};q^{-1}ab,ab^{-1};q)
\qquad(m,n\in\ZZ_{\ge0}).
\label{eq:dASCLqJ}
\end{equation}
This also follows by comparing the $q$-hypergeometric expression
in \eqref{eq:ASC} and \eqref{28}.
\begin{remark}
\label{160}
Formula \eqref{eq:dASCLqJ} should be compared with
\cite[Remark 3.1]{Groe}, where Al-Salam-Chihara polynomials
for base $q^{-1}$
are identified with dual little $q$-Jacobi polynomials. This identification is
less immediate from the $q$-hypergeometric expressions than
\eqref{eq:dASCLqJ} (it uses an additional $q$-hypergeometric
transformation formula), but it has the nice property that it also encodes a
duality of infinite discrete orthogonality relations.
It is also a limit case of a
similar identification (to be compared with our duality \eqref{69})
between dual big $q$-Jacobi polynomials and
continuous dual $q^{-1}$-Hahn polynomials, see \cite[Section 4.3]{AK}.
Again it needs an additional transformation formula for its derivation
and again it
encodes a duality of orthogonality relations. In fact, as pointed out in
\cite{AK}, this identification is a limit case for $N\to\iy$ of the duality
for $q$-Racah polynomials on a set of $N+1$ points.
On the level of the AW algebra and its limit cases
one should use, beside \eqref{70}, also \eqref{156} in connection with
the formulas just discussed.
\end{remark}
In order to take the limit for $c\to\iy$ in \eqref{69}, first
substitute $n\to N-n$ and $c\to q^{-N}\ga^{-1}$. Then
\begin{equation*}
R_{N-n}\big[a^{-1}q^{-m};a,b,-q^{-N}\ga^{-1}\,|\,q\big]=
P_m(q^{n-N};q^{-1}ab,ab^{-1},-q^{-N-1}\ga^{-1}a;q).
\end{equation*}
Here $m,N,n$ are integers with $m,N\ge0$ and $n\le N$.
Now apply \eqref{eq:l1} to the right-hand side
with $c=-q^{-N-1}\ga^{-1}a$ and $N\to\iy$ and apply \eqref{148}
to the left-hand side. We obtain a duality between
Askey-Wilson $q$-Bessel functions and little $q$-Jacobi polynomials:
\begin{equation*}
J_{q^n\ga}\big[a^{-1}q^{-m};a,b\,|\,q\big]=
p_m(-q^n\ga a^{-1};q^{-1}ab,ab^{-1};q).
\end{equation*}
\begin{remark}
\label{29}
In \eqref{eq:d-bqj-cdqh} replace $\la$ by $c^\half \la$:
\begin{equation*}
R_n\big[a^{-1}q^{-m};a,b,c,qa^{-1}b^{-1}\la^2\,|\,q\big]
=R_m\big[\la^{-1}q^{-n};c^\half\la,abc^{-\half}\la^{-1},ac^\half\la^{-1},
qb^{-1}c^{-\half}\la\,|\,q\big],
\end{equation*}
and let $(c,\la)\to(0,0)$.
Then we arrive directly from the \AW \/ duality at the duality
\eqref{eq:dASCLqJ}, as is seen from \eqref{62}, \eqref{eq:ASC} and
\eqref{28}. However, if we fix one of $c,\la$ at a nonzero value
and let the other one go to $0$ then we arrive at the duality \eqref{69}.
\end{remark}
Similarly to, and as a limit case of \eqref{57},
there is a duality corresponding to \eqref{eq:dASCLqJ} for the operators
$L_z$ and $M_n$ defined by \eqref{82} and \eqref{88}, respectively:
\begin{equation}
L_z^{a,b}\big(f[z]\big)\Big|_{z=a^{-1}q^{-m}}=
\td a\,M_m^{\td a,\td b}\big(f\big[a^{-1}q^{-m}\big]\big),
\label{61}
\end{equation}
where $L_z^{a,b}$ is the operator $L_z$  given by \eqref{82},
$M_n^{a,b}$ is the operator $M_n$ given by \eqref{88},
and
\begin{equation}
(\td a,\td b)=\big((qab)^\half,(qab^{-1})^\half\big).
\label{60}
\end{equation}
Note that the map $(a,b)\mapsto(\td a,\td b)$ is inverse to the
map $(a,b)\mapsto(q^{-1}ab ,ab^{-1})$.
\par
There is also a duality corresponding to \eqref{69} for the operators
$L_x$ and $M_n$ defined by \eqref{86} and \eqref{84}, respectively:
\begin{equation}
L_x^{a,b}\big(f(x)\big)\Big|_{x=q^{-m}}=
\td a\,M_m^{\td a,\td b}\big(f(q^{-m})\big),
\label{78}
\end{equation}
where $L_x^{a,b}$ is the operator $L_x$  given by \eqref{86},
$M_n^{a,b}$ is the operator $M_n$ given by \eqref{84},
and $\td a,\td b$ are as in \eqref{60}.
\subsection{Corresponding degenerations of the Zhedanov algebra and duality}
\label{118}
As $c\to 0$, the Zhedanov algebra
$\AW^{\rm CqDH}_{a,b,c;q}(3,Q_0;K_0,K_1)$, see \eqref{24},
tends to the algebra with two generators
$K_0$, $K_1$ and with relations \eqref{1} and \eqref{9}
where $C_1=D_1=0$ and $B,C_0,D_0,Q_0$ are defined as in
\eqref{30} with $c=0$:
\begin{equation}
\begin{split}
&B =(1-q^{-1})^2\big(ab+q(a+b)\big),\qquad
C_0 =(q-q^{-1})^2,\\
&D_0 =-q^{-2}(1-q)^2(1+q)(ab+q),\\
&Q_0=q^{-2}(1-q)^2\big(q(a^2+b^2)-(q^2+1)ab\big).
\end{split}
\label{31}
\end{equation}
In the expression \eqref{3} for $Q$ also put $C_1=0$ and
substitute \eqref{31}.
We denote this algebra by
$\AW^{\rm ASC}_{a,b;q}(3,Q_0;K_0,K_1)$.
\par
Similarly, as $c\to 0$, the Zhedanov algebra
$\AW^{\rm BqJ}_{a,b,c;q}(3,Q_0;K_0,K_1)$, see \eqref{27},
tends to the algebra with two
generators $K_0$, $K_1$ and with relations \eqref{1} and \eqref{9}
where $C_0=D_0=0$
and $B, C_1, D_1$ are given by formula \eqref{eq:z-BqJ}  with $c=0$,
and also $Q_0$ with $c=0$ has a simple expression:
\begin{equation}
\begin{split}
&B =(1-q)^2 a(b+1),\quad
C_1=q(q-q^{-1})^2 ab, \\& D_1=  - (1-q)^2(1+q)ab(a+1),\,
Q_0=(1-q)^2 a^2(b-q)(qb-1).
\end{split}
 \label{32}
\end{equation}
Furthermore $Q$ is now obtained by following the procedure
described for big $q$-Jacobi just before \eqref{eq:z-BqJ} and then
putting $c=0$.
We denote this algebra by $\AW^{\rm LqJ}_{a,b;q}(3,Q_0;K_0,K_1)$.
\par
The representations of the algebras
$\AW^{\rm ASC}_{a,b;q}(3,Q_0;K_0,K_1)$ and
$\AW^{\rm LqJ}_{a,b;q}(3,Q_0;K_0,K_1)$ can
be obtained from the representations of
$\AW_{a,b,c;q}^{\rm CDqH}(3,Q_0)$ and
$\AW_{a,b,c;q}^{\rm BqJ}(3,Q_0)$, respectively,
by putting $c=0$ in all formulae.
Similarly, the duality formula is then simply obtained
by putting $c=0$ in \eqref{con-70}:
\begin{proposition} There is an isomorphism (both an algebra
isomorphism and an algebra anti-isomorphism)
\begin{equation}
\AW^{\rm ASC}_{a,b;q}(3,Q_0;K_0,K_1)\simeq
\AW^{\rm LqJ}_{q^{-1}ab,ab^{-1};q}(3,Q_0;a K_1, K_0).
\label{con-70-c0}
\end{equation}
\end{proposition}

Just as in Remark \ref{29} we may replace $\la$ by $c^\half\la$ in
\eqref{eq:dual-l}, and then take the limit as $(c,\la)\to(0,0)$.
Then we will arrive at the duality \eqref{con-70} directly from the
duality of $\AW(3)$.
\par
The representation \eqref{700} of
$\AW^{\rm ASC}_{a,b;q}(3,Q_0;K_0,K_1)$
with $L_z$ given by \eqref{82}
corresponds under the map \eqref{con-70-c0}
to the representation \eqref{56} of
$\AW^{\rm LqJ}_{a,b;q}(3,Q_0;K_0,K_1)$
with $M_n$ given by \eqref{88}.
This follows by \eqref{61}.
\par
Similarly, the representation \eqref{19} of
$\AW^{\rm ASC}_{a,b;q}(3,Q_0;K_0,K_1)$ with $M_n$ and $\La_n$
given by \eqref{84} and \eqref{59}, respectively,
corresponds under the map \eqref{con-70-c0}
to the representation \eqref{55} of
$\AW^{\rm LqJ}_{a,b;q}(3,Q_0;K_0,K_1)$
wilt $L_x$ defined by \eqref{86}.
This follows by \eqref{78}.
\begin{remark}
Corresponding to the limit \eqref{148} we
can consider $\AW_{a,b,c;q}^{\rm CDqH}(3,Q_0)$ with structure
constants \eqref{30} and there raplace $K_0$ by $cK_0$.
Then by \eqref{30} the relations \eqref{1} have a limit for $c\to\iy$.
Similarly, corresponding to the limit \eqref{eq:l1} we can consider
$\AW_{a,b,c;q}^{\rm BqJ}(3,Q_0;K_0,K_1)$
with stucture constants \eqref{eq:z-BqJ} and there replace $K_1$ by $cK_1$.
Then by \eqref{eq:z-BqJ} the relations \eqref{1} have a limit for
$c\to\iy$. Representations of the algebras and duality could
be considered. We omit the details.
\end{remark}
\section{Askey-Wilson DAHA and nonsymmetric AW polynomials}
\label{152}
\subsection{Definition of the Askey-Wilson DAHA}
The DAHA of type $(C_1^\vee,C_1)$, also called the
{\em Askey-Wilson DAHA}, occurs as the rank one case of Sahi's more
general construction \cite{Sa}. It was studied in much detail by Noumi
and Stokman~\cite{NS}.

Keep the assumptions \eqref{eq:6} on $a,b,c,d$. We will notate the
Askey-Wilson DAHA as
$\DAHA=\DAHA_{a,b,c,d;q}(T_1,T_0,\check T_1,\check T_0)$,
and we define it as
the algebra with generators  $T_1,T_0,\check T_1,\check T_0$
and with relations
\begin{equation}
\begin{split}
(T_1+ab)(T_1+1)&=0,\\
(T_0+q^{-1}cd)(T_0+1)&=0,\\
(a\check T_1+1)(b\check T_1+1)&=0,\\
(c\check T_0+q)(d\check T_0+q)&=0,\\
\check T_1 T_1 T_0 \check T_0&=1.
\end{split}
\label{71}
\end{equation}
The relations \eqref{71}
imply that $T_1,T_0,\check T_1,\check T_0$ are invertible
{and that their inverses have simple expressions as linear
combinations of their original and a constant, in particular,
\begin{equation}
T_1^{-1}=-(ab)^{-1}T_1-\big((ab)^{-1}+1\big),\qquad
T_0^{-1}=-q(cd)^{-1}T_0-\big(q(cd)^{-1}+1\big).
\label{142}
\end{equation}}
\begin{remark}
In  \cite[Definition~2.1]{O} and \cite[Lemma 1.2]{M}
the Askey-Wilson DAHA is an algebra with generators
$V_1,V_0, \check V_1,\check V_0$ and relations
equivalent to \eqref{71} by the substitutions
$(a,b,c,d)=(-k_1u_1^{-1},k_1u_1,q^\half k_0u_0^{-1},-q^\half k_0u_0)$,\quad
$T_1=k_1V_1$, $\check T_1=k_1^{-1} \check V_1$,
$T_0=k_0V_0$, $\check T_0=q^\half k_0^{-1}\check V_0$.
\end{remark}
Put
\begin{equation*}
Z=T_0\check T_0,\quad\mbox{and thus}\quad Z^{-1}=\check T_1 T_1.
\end{equation*}
This means conversely that
\begin{equation*}
\check T_0=T_0^{-1}Z,\qquad
\check T_1=Z^{-1} T_1^{-1}.
\end{equation*}
With these substitutions in relations \eqref{71}
$\DAHA$ can be written equivalently as the algebra
with generators $T_1,T_0,Z,Z^{-1}$ and relations
\begin{equation}
\begin{split}
(T_1+ab)(T_1+1)&=0,\\
(T_0+q^{-1}cd)(T_0+1)&=0,\\
(aZ^{-1}T_1^{-1}+1)(bZ^{-1}T_1^{-1}+1)&=0,\\
(cT_0^{-1}Z+q)(dT_0^{-1}Z+q)&=0,\\
ZZ^{-1}=Z^{-1}Z&=1.
\end{split}
\label{110}
\end{equation}
We denote $\DAHA$ in this presentation by
$\DAHA_{a,b,c,d;q}[T_1,T_0,Z]$.
\begin{remark}
If the third and fourth relation in \eqref{110} are equivalently written as
$(T_1z+a)(T_1Z+b)=0$ and
$(qT_0Z^{-1}+c)(qT_0Z^{-1}+d)=0$, respectively, and if we put $X:=Z$
and $W:=Z^{-1}$ then we recover the relations (1.1)--(1.5) in \cite{M}.
\end{remark}
Replace in the relations \eqref{71}
the generators $T_0,\check T_1,\check T_0$ by $Z,Z^{-1},Y$
by putting
\begin{equation}
\check T_0=T_0^{-1}Z,\quad
\check T_1=Z^{-1}T_1^{-1},\quad
T_0=T_1^{-1}Y
\label{114}
\end{equation}
(or equivalently replace in the relations \eqref{110}
the generator $T_0$ by $Y$ by putting $T_0=T_1^{-1}Y$).
The substitutions \eqref{114} can be written conversely as
\begin{equation}
Z=T_0\check T_0,\quad
Z^{-1}=\check T_1 T_1,\quad
Y=T_1T_0.
\label{79}
\end{equation}
The resulting relations are
\begin{equation}
\begin{split}
(T_1+ab)(T_1+1)=0,\\
(T_1^{-1}Y+q^{-1}cd)(T_1^{-1}Y+1)=0,\\
(aZ^{-1}T_1^{-1}+1)(bZ^{-1}T_1^{-1}+1)=0,\\
(c+qZ^{-1}T_1^{-1}Y)(d+qZ^{-1}T_1^{-1}Y)=0,\\
ZZ^{-1}=1=Z^{-1}Z.
\end{split}
\label{75}
\end{equation}
We denote $\DAHA$ in this presentation by
$\DAHA_{a,b,c,d;q}\lan T_1,Y,Z^{-1}\ran$.
\subsection{Duality for the Askey-Wilson DAHA}
Note in \eqref{71} the trivial algebra isomorphism
\begin{equation*}
\DAHA_{a,b,c,d;q}(T_1,T_0,\check T_1,\check T_0)\simeq
\DAHA_{-a,-b,-c,-d;q}(T_1,T_0,-\check T_1,-\check T_0).
\end{equation*}
In terms of the generators in \eqref{110} this algebra isomorphism
is generated by
$(T_1,T_0,Z,Z^{-1})\mapsto(T_1,T_0,-Z,-Z^{-1})$,
and in terms of the generators in \eqref{75}
by $T_1,Y,Z^{-1}\mapsto T_1,Y,-Z^{-1}$.
In \eqref{71} we also recognize the following straightforward
algebra isomorphism:
\begin{equation}
\DAHA_{a,b,c,d;q}(T_1,T_0,\check T_1,\check T_0)\simeq
\DAHA_{a^{-1},b^{-1},c^{-1},d^{-1};q^{-1}}
(T_1^{-1},T_0^{-1},\check T_1^{-1},\check T_0^{-1}).
\label{157}
\end{equation}
\par
For the main duality property we need the dual parameters \eqref{64}.
Observe that
\begin{align*}
&ab=\td a\td b,\qquad\;\;
a+b=\big(q\td a\td b\td c^{-1}\td d^{-1}\big)^\half
(q^{-1}\td c\td d+1),\\
&cd=q\td a\td b^{-1},\quad
c+d=\big(q\td a\td b^{-1}\td c^{-1}\td d^{-1}\big)^\half
(\td c+\td d).
\end{align*}
As a consequence, there is an anti-isomorphism
\begin{equation}
\Phi\colon\DAHA_{a,b,c,d;q}(T_1,T_0,\check T_1,\check T_0)\simeq
\DAHA_{\td a,\td b,\td c,\td d;q}(T_1,\td a \check T_1,a^{-1} T_0,
a\td a^{-1}\check T_0).
\label{72}
\end{equation}
Indeed, substitution in relations \eqref{71}
of generators and parameters according to
\eqref{72} and reversion of the order of multiplication (only needed
in the last relation) interchanges the second and third relation,
while it preserves the other relations.
This also implies that the anti-isomorphism $\Phi$ is involutive.
\par
In terms of the presentation \eqref{110} we can write the
duality as
\begin{equation}
\Phi\colon \DAHA_{a,b,c,d;q}[T_1,T_0,Z]\simeq
\DAHA_{\td a,\td b,\td c,\td d;q}
[T_1,aZ^{-1}T_1^{-1},\td a T_0^{-1}T_1^{-1}],
\label{115}
\end{equation}
and in terms of the presentation \eqref{75}  as
\begin{equation}
\Phi\colon \DAHA_{a,b,c,d;q}\lan T_1,Y,Z^{-1}\ran\simeq
\DAHA_{\td a,\td b,\td c,\td d;q}\lan T_1,aZ^{-1},\td a^{-1} Y\ran.
\label{104}
\end{equation}
\par
Note also the following anti-isomorphism in terms of the presentation
\eqref{75}:
\begin{equation}
\DAHA_{a,b,c,d;q}\lan T_1,Y,Z^{-1}\ran \simeq
\DAHA_{a,b,c,d;q}\lan T_1,Y,T_1^{-1}Z^{-1}T_1\ran.
\label{132}
\end{equation}
\subsection{DAHA representation on the Laurent polynomials
and non-symmetric Askey Wilson polynomials}
\label{143}
The algebra $\DAHA$ in presentation \eqref{75}
has a faithful representation, the so-called
{\em basic representation}, on the space~$\FSA$
of Laurent polynomials $f[z]$ as follows (see \cite[\S3]{K} 
and \eqref{142} and use that $Y^{-1}=T_0^{-1}T_1^{-1}$):
\\
\begin{align}
(Zf)[z]&:=z\,f[z],
\label{14}
\\
\noalign{\allowbreak}
(T_1f)[z]&:=\frac{(a+b)z-(1+ab)}{1-z^2}\,f[z]+
\frac{(1-az)(1-bz)}{1-z^2}\,f[z^{-1}],
\label{15}
\\
\noalign{\allowbreak}
(T_0f)[z]&:=\frac{q^{-1}z\big((cd+q)z-(c+d)q\big)}{q-z^2}\,f[z]-
\frac{(c-z)(d-z)}{q-z^2}\,f[qz^{-1}],\nonumber
\\
\noalign{\allowbreak}
(T_1^{-1}f)[z]&=\frac{z\big((1+ab)z-(a+b)\big)}{ab(1-z^2)}\,f[z]
-\frac{(1-az)(1-bz)}{ab(1-z^2)}\,f[z^{-1}],\nonumber
\\
\noalign{\allowbreak}
(T_0^{-1}f)[z]&=\frac{q\big((c+d)z-(cd+q)\big)}{cd(q-z^2)}\,f[z]
+\frac{q(c-z)(d-z)}{cd(q-z^2)}\,f[qz^{-1}],\nonumber
\\
\noalign{\allowbreak}
(Yf)[z]&=
\frac{z \bigl(1+ab-(a+b)z\bigr)
\bigl((c+d)q-(cd+q)z\bigr)}{q(1-z^2)(q-z^2)}\,f[z]\nonumber\\
&\qquad+\frac{(1-az)(1-bz)(1-cz)(1-dz)}{(1-z^2)(1-q z^2)}f[qz]
\nonumber\\
&\qquad+\frac{(1-a z)(1-b z) \bigl((c+d)qz-(cd+q)\bigr)}
{q(1-z^2)(1-q z^2)}\,f[z^{-1}]\nonumber\\
&\qquad+\frac{(c-z)(d-z)\bigl(1+ab-(a+b)z\bigr)}{(1-z^2)(q-z^2)}\,f[qz^{-1}],
\label{34}
\\
\noalign{\allowbreak}
(Y^{-1}f)[z]&=
\frac{qz \big(a+b-(1+ab)z\big)\big(cd+q-(c+d)z\big)}
{abcd(1-z^2)(q-z^2)}\,f[z]\nonumber
\\
&\qquad+\frac{q(aq-z)(bq-z)(c-z)(d-z)}{abcd(q-z^2)(q^2-z^2)}
f[q^{-1}z]
\nonumber
\\
&\qquad+\frac{q(1-a z)(1-b z) \big(cd+q-(c+d)z\big)}
{abcd(1-z^2)(q-z^2)}\,f[z^{-1}]\nonumber
\\
&\qquad+\frac{q^2(c-z)(d-z)\big((a+b)z-q(1+ab)\big)}
{abcd(q-z^2)(q^2-z^2)}\,f[qz^{-1}].\label{144}
\end{align}
\par
Put
\begin{equation*}
D:= Y + q^{-1} a b c d Y^{-1}.
\end{equation*}
$D$ commutes with $T_1$, $T_0$ and $Y$.
If we compare its explicit expression \cite[(3.14)]{K} with
\eqref{eq:8} then we see that
\[
(Df)[z]=(Lf)[z]\quad\mbox{if}\quad f[z]=f[z^{-1}].
\]
In particular, if we apply $D$ to the AW polynomial
$R_n[z]$ given by \eqref{62} then we obtain that
\begin{equation*}
DR_n=\la_nR_n
\end{equation*}
with $\la_n$ given by \eqref{eq:22}.

More generally, see \cite[\S\S3,4]{K}, the eigenspace $\FSA_n$
of $D$ in $\FSA$ for eigenvalue $\la_n$ has dimension~2
if $n>0$ and dimension 1
if $n=0$. A basis of $\FSA_n$ is given by eigenfunctions $E_{\pm n}$
of $Y$.
These are the {\em  non-symmetric Askey-Wilson polynomials}
$E_{\pm n}[z;a,b,c,d\,|\,q]$, which we define
 by multiplying the ones given in
\cite[(4.9), (4.10)]{K} or (with different normalization) in
\cite[(4.2), (4.3)]{KB} by suitable factors such that their first
term becomes $R_n[z;a,b,c,d\,|\,q]$.
Then, in view of \cite[(2.10), (2.11), (3.17)]{K} or
\cite[(3.1)]{KB}, we get
\begin{equation}
\begin{split}
E_n[z;a,b,c,d\,|\,q]:=R_n[z;a,b,c,d\,|\,q]
-\frac{q^{1-n}(1-q^n)(1-q^{n-1}cd)}
{(1-qab)(1-ab)(1-ac)(1-ad)}\\
\times az^{-1} (1-az)(1-bz) R_{n-1}[z;qa,qb,c,d\,|\,q]\qquad
(n\ge 0),
\sLP
E_{-n}[z;a,b,c,d\,|\,q]:=R_n[z;a,b,c,d\,|\,q]
-\frac{q^{1-n}(1-q^n ab)(1-q^{n-1}abcd)}
{(1-qab)(1-ab)(1-ac)(1-ad)}\\
 \times b^{-1}z^{-1} (1-az)(1-bz) R_{n-1}[z;qa,qb,c,d\,|\,q]\qquad
(n\ge 1),
\end{split}
\label{40}
\end{equation}
where $(1-q^n)E_{n-1}:=0$ for $n=0$.
Then, by \cite[(4.4), (4.5)]{K},
\begin{equation}
\begin{split}
YE_n&=q^{n-1}abcd\,E_n\qquad(n=0,1,2,\ldots),
\\
YE_{-n}&=q^{-n}\,E_{-n}\qquad\quad\;\;(n=1,2,\ldots).
\end{split}
\label{25}
\end{equation}
\begin{remark}
\label{13}
The notations $P_m^+(x)$ and $P_m(x)$ in \cite[Theorem 5.9,
Proposition 5.10(i),(ii)]{NS} correspond to the notations in \cite{K} by
\[
P_m^+(x)=P_m[x] \;\;(m\in\ZZ_{\ge0}),\qquad
P_m(x)=E_m[x]\;\;(m\in\ZZ;\;\mbox{$E_m$ as in \cite{K}}).
\]
Indeed, compare \cite[Theorem 4.1]{K} with \cite[Proposition 5.10(i),(ii)]{NS}
while taking into account \cite[(3.19)]{K}.
Now see from \cite[\S\S3.3, 5.7, 10.6]{NS} that $x_0$ as defined in \cite{NS}
equals our $a$ and observe from \eqref{4}
that, for our $E_n$ in \eqref{40}, $E_n[a^{-1};a,b,c,d\,|\,q]=1$.
We conclude thar the renormalized  non-symmetric
AW polynomials in \cite[Definition 10.6(i)]{NS} have the same
normalization as in \eqref{40}. More explicitly:
\[
E_{\ga_m}(x;\underline t;q)=E_m[x;a,b,c,d\,\\,q],
\]
where the notation from \cite{NS} is on the left, where $\ga_m$ is defined
in \cite[\S3.4]{NS}, and where $\underline t$ is related to $a,b,c,d$ by
\cite[\S5.7]{NS}.
\end{remark}
\subsection{Duality of nonsymmetric AW polynomials}
First observe that the trivial symmetry \eqref{11} for AW polynomials
extends to a symmetry
\begin{equation}
E_n[z;a,b,c,d\,|\,q]=E_n[z^{-1};a^{-1},b^{-1},c^{-1},d^{-1}\,|\,q^{-1}]\qquad
(n\in\ZZ)
\label{125}
\end{equation}
for non-symmetric AW polynomials. This is clear from \eqref{40} and
\eqref{11}. Compare also with the DAHA
algebra isomorphism \eqref{157}.
Next we pass to the main duality result.
\begin{theorem}{\rm \cite[Theorem 10.7(i)]{NS}}\quad
\label{th:11}
Let
\begin{equation}
\begin{split}
z_{a,q}(n)&:=aq^n\quad(n\in\ZZ_{\ge0}),\\
z_{a,q}(-n)&:=a^{-1}q^{-n}\quad(n\in\ZZ_{>0}).
\end{split}
\label{26}
\end{equation}
Then
\begin{equation}
E_{n}\big[z_{a,q}(m)^{-1};a,b,c,d\,|\,q\big]=
E_m\big[z_{\td a,q}(n)^{-1};\td a,\td b,\td c,\td d\,|\,q\big]\qquad
(m,n\in\ZZ).
\label{eq:15}
\end{equation}
\end{theorem}
\begin{proof}
For the case that $n$ or $m=0$ use that
$E_n[a^{-1};a,b,c,d\,|\,q]=1$.
For the other cases we have to show that
\begin{align}
E_{-n}\big[a^{-1}q^{-m};a,b,c,d\,|\,q\big]&=
E_m\big[\td a q^n;\td a,\td b,\td c,\td d\,|\,q\big],
\label{eq:7}\\
E_{-n}\big[aq^m;a,b,c,d\,|\,q\big]&=
E_{-m}\big[\td a q^n;\td a,\td b,\td c,\td d\,|\,q\big],
\label{eq:80}\\
E_n\big[a^{-1}q^{-m};a,b,c,d\,|\,q\big]&=
E_m\big[\td a^{-1} q^{-n};\td a,\td b,\td c,\td d\,|\,q\big].
\label{eq:9}
\end{align}
for $m,n\in\ZZ_{>0}$.
From \eqref{66} we see that
\begin{equation}
R_{n-1}[a^{-1}q^{-m};qa,qb,c,d\,|\,q]
=R_{m-1}\big[\td a^{-1} q^{-n};q\td a,q\td b,\td c,\td d\,|\,q\big].
\label{eq:60}
\end{equation}
By \eqref{40}, \eqref{66}, \eqref{eq:60} and by the identities
$ab=\td a\td b$, $ac=\td a\td c$, $ad=\td a\td d$ we see that
\eqref{eq:7}, \eqref{eq:80}, \eqref{eq:9} will respectively follow from the
identities
\begin{align*}
&q^{-n}(1-q^n ab)(1-q^{n-1}abcd)q^{-m}(1-q^m)(1-q^m ab^{-1})\\
&\qquad
=q^{-m}(1-q^m)(1-q^{m-1}\td c\td d)q^{-n}(1-q^n\td a^2)
(1-q^n\td a\td b),\sLP
&q^{-n}(1-q^n ab)(1-q^{n-1}abcd)a^{-1}b^{-1}q^{-m}(1-a^2q^m)(1-abq^m)\\
&\qquad
=q^{-m}(1-q^m \td a\td b)(1-q^{m-1}\td a\td b\td c\td d)\td a^{-1}
\td b^{-1}q^{-n}(1-\td a^2q^n)(1-\td a\td bq^n),\sLP
&q^{-n}(1-q^n)(1-q^{n-1}cd)a^2q^m(1-q^{-m})(1-a^{-1}b^{-1}q^{-m})\\
&\qquad
=q^{-m}(1-q^m)(1-q^{m-1}\td c\td d)\td a^2q^n(1-q^{-n})
(1-\td a^{-1}\td b^{-1}q^{-n})
\end{align*}
These indeed hold by \eqref{64}.
\end{proof}
\begin{remark}
In view of Remark \ref{13} our formula \eqref{eq:15} matches
with the formula
\[
E_{\ga_m}(x_n^{-1};\underline t;q)=E_{x_n}(\ga_m^{-1};\td{\underline t};q)
\]
in \cite[Theorem 10.7(i)]{NS} if we also take into account the formulas
for $\ga_m$ and  $x_m$ in \cite[\S\S3.4, 10.6]{NS}.
However note that the definition of $E_m$ in \cite{NS} involving
\cite[Proposition 5.10(i),(ii)]{NS}, i.e., AW polynomials with
parameters
$q^\half a, q^\half b, q^\half c, q^\half d$, does not allow direct
explicit verification of the duality formula \eqref{eq:15} because
AW polynomials depending on such parameters
are dually related to AW polynomials of parameters $qa,b,c,d$
by \eqref{66}.
\end{remark}
\subsection{Recurrence relation for the nonsymmetric AW
polynomials}
\label{108}
The duality \eqref{eq:15} for non-symmetric AW polynomials
$E_n$ ($n\in\ZZ$) can be applied to the eigenvalue equation \eqref{25}
in order to obtain a recurrence relation for the $E_n$.
First we consider $(YF)[z]$, given by \eqref{34}, at
$z=z_{a,q}(m)^{-1}$ ($m\in\ZZ$), with $z_{a,q}(m)$ given by \eqref{26}.
Note that
\begin{align*}
q z_{a,q}(m)^{-1}&=z_{a,q}(m-1)^{-1}\quad(m\ne0),\\
z_{a,q}(m)&=z_{a,q}(-m)^{-1}\qquad(m\ne0),\\
qz_{a,q}(m)&=z_{a,q}(-m-1)^{-1}.
\end{align*}
Note also that the terms in \eqref{34} with $f[qz]$ and with $f[z^{-1}]$
vanish if $z=a^{-1}=z_{a,q}(0)^{-1}$.
Thus we can specialize \eqref{34} as follows:
\begin{equation}
(Yf)\big[z_{a,q}(m)^{-1}\big]=
Af\big[\tfrac1{z_{a,q}(m)}\big]+Bf\big[\tfrac1{z_{a,q}(m-1)}\big]
+Cf\big[\tfrac1{z_{a,q}(-m)}\big]+Df\big[\tfrac1{z_{a,q}(-m-1)}\big]
\label{1000}
\end{equation}
with
\begin{align*}
A&=
\frac{\bigl(1+ab-(a+b)z_{a,q}(m)^{-1}\bigr)
\bigl((c+d)q-(cd+q)z_{a,q}(m)^{-1}\bigr)}
{qz_{a,q}(m)(1-z_{a,q}(m)^{-2})(q-z_{a,q}(m)^{-2})}\,,
\\
\noalign{\allowbreak}
B&=\frac{(1-az_{a,q}(m)^{-1})(1-bz_{a,q}(m)^{-1})(1-cz_{a,q}(m)^{-1})
(1-dz_{a,q}(m)^{-1})}
{(1-z_{a,q}(m)^{-2})(1-qz_{a,q}(m)^{-2})}\,,
\\
\noalign{\allowbreak}
C&=\frac{(1-az_{a,q}(m)^{-1})(1-b z_{a,q}(m)^{-1})
\bigl((c+d)qz_{a,q}(m)^{-1}-(cd+q)\bigr)}
{q(1-z_{a,q}(m)^{-2})(1-qz_{a,q}(m)^{-2})}\,,\\
\noalign{\allowbreak}
D&=\frac{(c-z_{a,q}(m)^{-1})(d-z_{a,q}(m)^{-1})
\bigl(1+ab-(a+b)z_{a,q}(m)^{-1}\bigr)}
{(1-z_{a,q}(m)^{-2})(q-z_{a,q}(m)^{-2})}\,,
\end{align*}
where the second and third term on the right in \eqref{1000}
vanish if $m=0$.
Now take $f=E_n$ and observe that the eigenvalue equation \eqref{25}
can be written in a unified way as
\begin{equation}
(YE_n)[z]=\td a\,z_{\td a,q}(n)\,E_n[z]\quad(n\in\ZZ)
\label{80}
\end{equation}
Thus, by the duality \eqref{eq:15} we can rewrite \eqref{80}
for $z=z_{a,q}(m)^{-1}$ as
\begin{multline}
\td a\,z_{\td a,q}(n)\,\td E_m\big[z_{\td a;q}(n)^{-1}\big]
=A\td E_m\big[z_{\td a;q}(n)^{-1}\big]
+BE_{m-1}\big[z_{\td a;q}(n)^{-1}\big]\\
+C\td E_{-m}\big[z_{\td a;q}(n)^{-1}\big]
+D\td E_{-m-1}\big[z_{\td a;q}(n)^{-1}\big],
\label{90}
\end{multline}
where
\begin{align*}
A&=\frac{\bigl(1+ab-(a+b)z_{a,q}(m)^{-1}\bigr)
\bigl((c+d)q-(cd+q)z_{a,q}(m)^{-1}\bigr)}
{qz_{a,q}(m)(1-z_{a,q}(m)^{-2})(q-z_{a,q}(m)^{-2})}\,,
\\
B&=\frac{(1-az_{a,q}(m)^{-1})(1-bz_{a,q}(m)^{-1})(1-cz_{a,q}(m)^{-1})
(1-dz_{a,q}(m)^{-1})}
{(1-z_{a,q}(m)^{-2})(1-qz_{a,q}(m)^{-2})}\,,
\\
C&=\frac{(1-az_{a,q}(m)^{-1})(1-b z_{a,q}(m)^{-1})
\big((c+d)qz_{a,q}(m)^{-1}-(cd+q)\big)}
{q(1-z_{a,q}(m)^{-2})(1-qz_{a,q}(m)^{-2})}\,,
\\
D&=\frac{(c-z_{a,q}(m)^{-1})(d-z_{a,q}(m)^{-1})
\bigl(1+ab-(a+b)z_{a,q}(m)^{-1}\bigr)}
{(1-z_{a,q}(m)^{-2})(q-z_{a,q}(m)^{-2})}\,,
\end{align*}
and where $\td E_m$ means $E_m$ for dual parameters.
Because we are dealing with Laurent polynomials, the equality \eqref{90}
will remain valid if we replace $z_{\td a;q}(n)^{-1}$ by arbitrary complex
$z$. Next, replace in \eqref{90} $a,b,c,d$ by $\td a,\td b,\td c,\td d$,
and replace $m$ by $n$. We obtain:
\begin{align*}
&az^{-1} E_n[z]\\
&=\frac{\bigl(1+\td a\td b-(\td a+\td b)z_{\td a,q}(n)^{-1}\bigr)
\bigl((\td c+\td d)q-(\td c\td d+q)z_{\td a,q}(n)^{-1}\bigr)}
{qz_{\td a,q}(n)(1-z_{\td a,q}(n)^{-2})(q-z_{\td a,q}(n)^{-2})}\,
E_n[z]\\
&\qquad
+\frac{(1-\td az_{\td a,q}(n)^{-1})(1-\td bz_{\td a,q}(n)^{-1})
(1-\td cz_{\td a,q}(n)^{-1})(1-\td dz_{\td a,q}(n)^{-1})}
{(1-z_{\td a,q}(n)^{-2})(1-qz_{\td a,q}(n)^{-2})}\,
E_{n-1}[z]\\\
&\qquad+\frac{(1-\td az_{\td a,q}(n)^{-1})(1-\td b z_{\td a,q}(n)^{-1})
\bigl((\td c+\td d)qz_{\td a,q}(n)^{-1}-(\td c\td d+q)\bigr)}
{q(1-z_{\td a,q}(n)^{-2})(1-qz_{\td a,q}(n)^{-2})}\,
E_{-n}[z]\\
&\qquad+\frac{(\td c-z_{\td a,q}(n)^{-1})(\td d-z_{\td a,q}(n)^{-1})
\bigl(1+\td a\td b-(\td a+\td b)z_{\td a,q}(n)^{-1}\bigr)}
{(1-z_{\td a,q}(n)^{-2})(q-z_{\td a,q}(n)^{-2})}\,
E_{-n-1}[z].
\end{align*}
Finally put
\begin{equation}
\nu_{\td a ,q}(n):=\td a^{-1} z_{\td a,q}(n)^{-1}=
\begin{cases}
(abcd)^{-1}q^{-n+1},&n\ge0,\\
q^{-n},&n<0.
\end{cases}
\label{103}
\end{equation}
Then we obtain the recurrence relation for the nonsymmetric
AW polynomials:
\begin{equation}
M_n\big(E_n[z]\big)=z^{-1} E_n[z],
\label{101}
\end{equation}
where $M_n$ is an operator acting on functions $g(n)$ of $n$
($n\in\ZZ$) which is given by
\begin{align}
&M_n\big(g(n)\big):=\nonu\\
&\frac{\nu_{\td a ,q}(n)(1+ab-ab(q^{-1}cd+1)\nu_{\td a ,q}(n))
(q(c+d)-cd(a+b)\nu_{\td a ,q}(n))}
{(q-abcd\nu_{\td a ,q}(n)^2)(q-q^{-1}abcd\nu_{\td a ,q}(n)^2)}\,
g(n)\nonu\\
&+\frac{(1-q^{-1}abcd\nu_{\td a ,q}(n))(1-ab\nu_{\td a ,q}(n))
(1-ac\nu_{\td a ,q}(n))(1-ad\nu_{\td a ,q}(n))}
{a(1-q^{-1}abcd\nu_{\td a ,q}(n)^2)(1-abcd\nu_{\td a  ,q}(n)^2)}\,
g(n-1)\nonu\\
&+\frac{(1-q^{-1}abcd\nu_{\td a ,q}(n))(1-ab\nu_{\td a ,q}(n))
(ab(c+d)\nu_{\td a ,q}(n)-(a+b))}
{ab(1-q^{-1}abcd\nu_{\td a ,q}(n)^2)(1-abcd\nu_{\td a  ,q}(n)^2)}\,
g(-n)\nonu\\
&+\frac{q^2(1-q^{-1}bc\nu_{\td a ,q}(n))
(1-q^{-1}bd\nu_{\td a ,q}(n))
(1+ab-ab(q^{-1}cd+1)\nu_{\td a ,q}(n))}
{b(q-abcd\nu_{\td a ,q}(n)^2)(q-q^{-1}abcd\nu_{\td a  ,q}(n)^2)}\,
g(-n-1).
\label{102}
\end{align}
\par
Note that by the symmetry \eqref{125} there also follows a recurrence
formula which expands $z E_n[z]$. We omit the explicit expression.
\begin{remark}
In \cite[\S10.9]{NS} it is just observed that
a recurrence relation for nonsymmetric AW polynomials can
be derived from the $Y$-eigenvalue equation by duality, but no further
derivation or explicit formula is given. Neither we have found such
a formula elsewhere in the literature.
\end{remark}
\subsection{The dual of the basic representation}
\label{137}
It follows from the derivation of \eqref{101}, \eqref{102}
in \S\ref{108} that
\begin{equation}
(Yf)[z_{a,q}(m)^{-1}]=
\td a\,\td M_m\big(f[z_{a,q}(m)^{-1}]\big),
\label{105}
\end{equation}
where, as usual, $\td M_m$ means the operator $M_m$ with respect to
dual parameters. Define a multiplication operator $N_m$, acting
on functions $g(m)$ ($m\in\ZZ$), by
\begin{equation}
N_m\big(g(m)\big):=\nu_{\td a ,q}(m)^{-1}\,g(m).
\label{136}
\end{equation}
Then, by \eqref{103},
\begin{equation}
(Z^{-1}f)[z_{a,q}(m)^{-1}]=a^{-1}\,\td N_m\big(f[z_{a,q}(m)^{-1}]\big).
\label{106}
\end{equation}
Also, in terms of $T_1$ acting on $f[z]$ by \eqref{15},
let $\FST_m$ be the operator acting on $g(m)$ such that
\begin{equation}
(T_1f)[z_{a,q}(m)^{-1}]=
\td \FST_m\big(f[z_{a,q}(m)^{-1}]\big).
\label{107}
\end{equation}
Since $T_1,Y,Z^{-1}$ acting on $f[z]$ satisfy the relations \eqref{75},
the same is true by \eqref{105}, \eqref{106}, \eqref{107}
for $\td \FST_m$, $\td a \td M_m$, $a^{-1}\td N_m$ acting on $g(m)$.
By the duality \eqref{104} $\td \FST_m,\td N_m,\td M_m$ satisfy
the relations \eqref{75} with respect to dual parameters and they
generate an anti-representation of $\DAHA$ with respect to dual
parameters. Hence we have also an anti-representation of
$\DAHA_{a,b,c,d;q}\lan T_1,Y,Z^{-1}\ran$ generated by
$T_1\to \FST_m$, $Y\to N_m$, $Z^{-1}\to M_m$.
\section{The basic representation
of the Askey-Wilson DAHA in a 2D realization}
\subsection{Definitions and explicit formulas}
We use results and notation from
\cite[(4.7) and following]{KB} except for a slight rescaling:
in \eqref{109} below we have an additional factor $a$ in the second
term on the right, which will also have impact on formulas further down.
This will facilitate the limit to Big $q$-Jacobi, which we will
consider later in the paper.
\par
The set-up is to associate with a Laurent polynomial $f$
a column vector $\vv{f}=\bma f_1\\f_2\ema$,
where $f_1,f_2$ are symmetric Laurent polynomials such that
\begin{equation}
f[z]=f_1[z]+az^{-1}(1-az)(1-bz)f_2[z].
\label{109}
\end{equation}
Then
\begin{equation}
\begin{split}
f_1[z]&=\frac{(z-a)(z-b)}{(ab-1)(1-z^2)}\,f[z]
-\,\frac{(1-az)(1-bz)}{(ab-1)(1-z^2)}\,f[z^{-1}],\\
f_2[z]&=\frac1{a(ab-1)}\,\frac{f[z]-f[z^{-1}]}{z-z^{-1}}\,.
\end{split}
\label{126}
\end{equation}
Put
\begin{equation}
\textstyle
\begin{split}
&\mathbf{S}:=\bma 1&\frac{a(1-az)(1-bz)}z\mLP
1&\frac{a(a-z)(b-z)}z\ema,\\
&\mathbf{S}^{-1}=
\frac1{(1-ab)(z-z^{-1})}
\bma \frac{(1-az)(1-bz)}z&
-\frac{(a-z)(b-z)}z\mLP-a^{-1}&a^{-1}\ema.
\end{split}
\label{127}
\end{equation}
Then \eqref{109} and \eqref{126} can be written more succinctly as
\begin{equation}
\bma f[z]\sLP f[z^{-1}]\ema=\mathbf{S}\,\vv{f}[z],\qquad
\vv{f}[z]=\mathbf{S}^{-1}\bma f[z]\sLP f[z^{-1}]\ema.
\label{129}
\end{equation}
\par
The non-symmetric AW polynomials $E_{\pm n}[z]$, as defined in
\eqref{40}, already have the decomposition \eqref{109}.
Thus they have vector-valued form
(see also \cite[(4.10), (4.11)]{KB})
\begin{equation}
\begin{split}
\vv{E_n}[z]&=
\bma R_n[z;a,b,c,d\mid q]\sLP
\displaystyle-\frac{\si(n)R_{n-1}[z;qa,qb,c,d\mid q]}
{(1-qab)(1-ab)(1-ac)(1-ad)}\ema\quad
(n=0,1,2,\ldots),\\
\vv{E_{-n}}[z]&=
\bma R_n[z;a,b,c,d\mid q]\sLP\displaystyle
 -\frac{\si(-n)R_{n-1}[z;qa,qb,c,d\,|\,q]}
{(1-qab)(1-ab)(1-ac)(1-ad)}\ema
\quad(n=1,2,\ldots),
\end{split}
\label{20}
\end{equation}
where
\begin{equation}
\begin{split}
\si(n)&:=q^{1-n}(1-q^n)(1-q^{n-1}cd)\qquad\qquad(n=0,1,2,\ldots),\\
\si(-n)&:=(ab)^{-1}q^{1-n}(1-q^n ab)(1-q^{n-1}abcd)\qquad(n=1,2,\ldots),
\end{split}
\label{124}
\end{equation}
and where $\si(n)R_{n-1}={\rm const.}\,(1-q^n)R_{n-1}:=0$ for $n=0$.
\par
Let $A$ be an operator acting on the space of Laurent polynomials.
Then we can write
\begin{equation}
(Af)[z]=
(A_{11}f_1+A_{12}f_2)[z]+az^{-1}(1-az)(1-bz)
(A_{21}f_1+A_{22}f_2)[z],
\label{123}
\end{equation}
where the $A_{i\,j}$ are operators acting on the space of symmetric
Laurent polynomials.
So we have the identifications
\begin{equation*}
f\leftrightarrow\bma f_1\\f_2\ema=\vv{f},\quad
A\leftrightarrow\bma A_{11}&A_{12}\\A_{21}&A_{22}\ema=\mathbf{A}.
\end{equation*}
In the identification $A\leftrightarrow \mathbf{A}$ composition of operators
corresponds to matrix multiplication together with entrywise composition
of operators. Indeed, we can express \eqref{123} as
$\vv{Af}=\mathbf{A}\vv{f}$. Then
\[
\vv{(AB)(f)}=\vv{A(Bf)}=\mathbf{A}\,\vv{Bf}=\mathbf{A}\mathbf{B}\vv{f}.
\]
\par
In this way $T_1$, given by \eqref{15},
acts as a $2\times 2$ matrix-valued operator:
\begin{equation}
\mathbf{T_1}=\begin{pmatrix}-ab&0\mLP 0&-1\end{pmatrix}.
\label{7}
\end{equation}
The very simple form of $\mathbf{T_1}$ as a diagonal matrix with
constant coefficients was the motivation for the decomposition
\eqref{109}.
\par
Next we describe the $2\times 2$ matrix-valued operator
\begin{equation}
\mathbf{Y}=\bma Y_{11}&Y_{12}\\Y_{21}&Y_{22}\ema.
\label{22}
\end{equation}
corresponding to $Y$ given by \eqref{34}.
Below we give the explicit expressions for the $Y_{ij}$,
see (4.12)--(4.15) in \cite{KB}.
The expressions  for $Y_{11}$ and $Y_{22}$ involve the
operator $L=L_{a,b,c,d;q}$  as given in \eqref{eq:8}:
\begin{align}
Y_{11}&=\frac{ab(q^{-1}cd+1)-ab L_{a,b,c,d;q}}{1-ab}\,,
\label{47}\\
Y_{22}&=\frac{-ab(q^{-1}cd+1)+q^{-1}L_{aq,bq,c,d;q}}{1-ab}\,,
\label{48}
\end{align}
We give $Y_{12}$ and $Y_{21}$ as operators
acting on a symmetric Laurent polynomial $g[z]$:
\begin{align}
&(Y_{21}g)[z]=
\frac{z(c-z)(d-z)}{a(1-ab)(1-z^2)(q-z^2)}\,\big(g[q^{-1}z]-g[z]\big)
\nonu\\
&\qquad\qquad\qquad
+\frac{z(1-cz)(1-dz)}{a(1-ab)(1-z^2)(1-qz^2)}\,\big(g[qz]-g[z]\big),
\label{49}\\
&(Y_{12}g)[z]=
\frac{a^2b(a-z)(b-z)(1-az)(1-bz)}{(1-ab)z(q-z^2)(1-qz^2)}\times\big((cd+q)(1+z^2)-(1+q)(c+d)z\big)g[z]
\nonu\\
\noalign{\allowbreak}
&\qquad\qquad\qquad
-\frac{a^2b(a-z)(b-z)(c-z)(d-z)(aq-z)(bq-z)}{q(1-ab)z(1-z^2)(q-z^2)}\,
g[q^{-1}z]
\nonu\\
\noalign{\allowbreak}
&\qquad\qquad\qquad
-\frac{a^2b(1-az)(1-bz)(1-cz)(1-dz)(1-aqz)(1-bqz)}
{q(1-ab)z(1-z^2)(1-qz^2)}\,g[qz].
\label{50}
\end{align}
Note an error in the formula \cite[(4.14)]{KB} for $(Y_{21}g)[z]$,
which we have corrected above: In the first term on the right we have
replaced the denominator factor $1-qz^2$ by $q-z^2$.
\par
Then the eigenvalue equation \eqref{25}, with $Y$ in
matrix-valued form and $E_{\pm n}$ in vector-valued form as given above,
still holds:
\begin{equation}
\begin{split}
\mathbf{Y}\vv{E_n}&=q^{n-1}abcd\,\vv{E_n}
\qquad\;(n=0,1,2,\ldots),\\
\mathbf{Y} \vv{E_{-n}}&=q^{-n}\,\vv{E_{-n}}
\qquad\qquad(n=1,2,\ldots).
\end{split}
\label{21}
\end{equation}

{By \cite[(3.7), (3.6)]{K} we can express $Y^{-1}$
in terms of $Y$ and $T_1^{-1}$. Indeed,
\begin{align}
Y^{-1}&=T_0^{-1}T_1^{-1}
=-qc^{-1}d^{-1}T_0T_1^{-1}-(1+qc^{-1}d^{-1})T_1^{-1}
\nonu\\
&=-qc^{-1}d^{-1}T_1^{-1}YT_1^{-1}-(1+qc^{-1}d^{-1})T_1^{-1}.
\label{46}
\end{align}
Then the matrix realization of $Y^{-1}$ follows from \eqref{46},
\eqref{47}--\eqref{50} and \eqref{7}.}
\par
As a final example the multiplication operator $Z$, given by \eqref{14},
corresponds to a $2\times2$ matrix-valued operator $\mathbf{Z}$ with
matrix entries acting as multiplication operators:
\begin{equation}
\mathbf{Z}=\frac1{ab-1}
\bma a+b-z-z^{-1}&-a(1-az)(1-az^{-1})(1-bz)(1-bz^{-1})\mLP
a^{-1}&ab(z+z^{-1})-(a+b)\ema.
\label{5}
\end{equation}
Note that $\det(\mathbf{Z})=1$. Hence
\begin{equation}
\mathbf{Z}^{-1}=\frac1{ab-1}
\bma ab(z+z^{-1})-(a+b)&a(1-az)(1-\frac{a}{z})(1-bz)(1-\frac{b}{z})\mLP
-a^{-1}&a+b-z-z^{-1}\ema.
\label{6}
\end{equation}
$\mathbf{Z}^{-1}$ can be diagonalized by
\begin{equation}
\mathbf{Z}^{-1}=\mathbf{S}^{-1}\bma z^{-1}&0\\0&z\ema\mathbf{S},
\label{128}
\end{equation}
where $\mathbf{S}$ and $\mathbf{S}^{-1}$ are given in \eqref{127}.
\par
Now consider \eqref{129} for $f:=E_n$:
\begin{equation*}
\bma E_n[z]\sLP E_n[z^{-1}]\ema=\mathbf{S}\,\vv{E_n}[z],\qquad
\vv{E_n}[z]=\mathbf{S}^{-1}\bma E_n[z]\sLP E_n[z^{-1}]\ema.
\end{equation*}
On combination with \eqref{128} and \eqref{101} this shows that
$\vv{E_n}[z]$ satisfies a similar recurrence relation as $E_n[z]$, namely
\begin{equation}
M_n\Big(\vv{E_n}[z]\Big)=\mathbf{Z}^{-1} \vv{E_n}[z],
\label{134}
\end{equation}
where the operator $M_n$, acting on functions of $n$, is given by
\eqref{102}.
\subsection{Orthogonality and equivalence of representations}
In addition to the conditions on $a,b,c,d,q$ at the end of \S\ref{133}
assume that $a,b$ are real and $ab<0$. Put
\begin{equation*}
w_{a,b,c,d;q}[z]:=
\frac{(q,ab,ac,ad,bc,bd,cd;q)_\iy}{4\pi (abcd;q)_\iy}
\left|\frac{(z^2;q)_\iy}{(az,bz,cz,dz;q)_\iy}\right|^2
\end{equation*}
for the weight function in the orthogonality relations \eqref{131}
for the AW polynomials,
and put
\begin{equation*}
\mathbf{W[z]}:=\bma w_{a,b,c,d;q}[z]&0\\0&
C_{a,b,c,d;q}w_{qa,qb,c,d;q}[z]\ema,
\end{equation*}
where
\begin{equation*}
C_{a,b,c,d;q}:=-a^3b\,\frac{(1-ab)(1-qab)(1-ac)(1-ad)(1-bc)(1-bd)}
{(1-abcd)(1-qabcd)}>0.
\end{equation*}
Then, by \cite[\S5]{KB}, we have the following orthogonality
relations with respect to a positive inner product:
\begin{equation}
\int_{|z|=1} \Big(\vv{E_m}[z]\Big)^{\rm t}\,\mathbf{W}[z]\,\vv{E_n}[z]\,
\frac{dz}{iz}=h_n\,\de_{m,n}\qquad(m,n\in\ZZ)
\label{138}
\end{equation}
for certain $h_n>0$. Here $\vv{v}^{\,\rm t}$ means the column vector
$\vv{v}$ written as a row vector, and more generally
$\mathbf{A}^{\rm t}$ means the transpose of a matrix $\mathbf{A}$.
\par
Now observe that
\begin{align*}
&(ab -1)\mathbf{W}[z]^{-1}
\big(\mathbf{T}_1(\mathbf{Z}^{-1})^{\rm t}\,\mathbf{T}_1^{-1}\big)
\mathbf{W}[z]\\
&\quad=\,\mathbf{W}[z]^{-1}
\bma ab(z+z^{-1})-(a+b)&-(a^2b)^{-1}\sLP
a^2b(1-az)(1-
\frac{a}{z})(1-bz)(1-\frac{b}{z})&a+b-z-\frac{1}{z}\ema \mathbf{W}[z]\\
&\quad=
\bma ab(z+z^{-1})-(a+b)&a(1-az)(1-\frac{a}{z})(1-bz)(1-\frac{b}{z})\mLP
-a^{-1}&a+b-z-z^{-1}\ema=(ab -1)\mathbf{Z}^{-1},
\end{align*}
since
\begin{equation*}
C_{a,b,c,d;q}\,\frac{w_{qa,qb,c,d;q}[z]}{w_{a,b,c,d;q}[z]}
=-a^3b(1-az)(1-az^{-1})(1-bz)(1-bz^{-1}).
\end{equation*}
Hence, for 2-vector valued Laurent polynomials $\vv{f}[z],\vv{g}[z]$
we have
\begin{equation}
\int_{|z|=1}
\Big(\mathbf{T}_1^{-1}\mathbf{Z}^{-1}\mathbf{T}_1\vv{f}[z]\Big)^{\rm t}
\mathbf{W}[z]\,\vv{g}[z]\,\frac{dz}{iz}
=\int_{|z|=1}\big(\vv{f}[z]\big)^{\rm t}\,\mathbf{W}[z]\,
\mathbf{Z}^{-1} \vv{g}[z]\,\frac{dz}{iz}.
\label{135}
\end{equation}
\par
Now we can show, analogous to \eqref{39} and following, that
we can use $\vv{E}_n[z]$ in order to pass from the basic representation
of $\DAHA$ to an anti-representation of $\DAHA$ acting on
scalar-valued functions of $n$. Define a Fourier-type transform
\begin{equation*}
\big(\vv{f}\big){}\hat{\vphantom{1}}\,(n):=
\int_{|z|=1} \big(\vv{f}[z]\big)^{\rm t}\,\mathbf{W}[z]\,
\vv{E_n}[z]\,\frac{dz}{iz}\qquad(n\in\ZZ).
\end{equation*}
Then, by \eqref{134} and \eqref{135},
\begin{align*}
M_n\Big(\big(\vv{f}\big){}\hat{\vphantom{1}}\,(n)\Big)
=\int_{|z|=1} \big(\vv{f}[z]\big)^{\rm t}\,\mathbf{W}[z]\,
M_n\big(\vv{E_n}[z]\big)\,\frac{dz}{iz}
=\int_{|z|=1} \big(\vv{f}[z]\big)^{\rm t}\,\mathbf{W}[z]\,
\mathbf{Z}^{-1}\vv{E_n}[z]\,\frac{dz}{iz}\\
=\int_{|z|=1}
\Big(\mathbf{T}_1^{-1}\mathbf{Z}^{-1}\mathbf{T}_1\vv{f}[z]\Big)^{\rm t}
\mathbf{W}[z]\,\vv{E_n}[z]\,\frac{dz}{iz}
=\Big(\mathbf{T}_1^{-1}\mathbf{Z}^{-1}\mathbf{T}_1\vv{f}\Big)
{}\hat{\vphantom{1}}\,(n).
\end{align*}
\par
Furthermore, by \eqref{103}, \eqref{136} and \eqref{21}
we have
\begin{align*}
N_n\Big(\big(\vv{f}\big){}\hat{\vphantom{1}}\,(n)\Big)
=\int_{|z|=1} \big(\vv{f}[z]\big)^{\rm t}\,\mathbf{W}[z]\,
N_n\big(\vv{E_n}[z]\big)\,\frac{dz}{iz}
=\int_{|z|=1} \big(\vv{f}[z]\big)^{\rm t}\,\mathbf{W}[z]\,
\big(\mathbf{Y}\vv{E_n}\big)[z]\,\frac{dz}{iz}\\
=\int_{|z|=1} \big(\mathbf{Y}\vv{f}[z]\big)^{\rm t}\,\mathbf{W}[z]\,
\vv{E_n}[z]\frac{dz}{iz}
=\Big(\mathbf{Y}\vv{f}\Big){}\hat{\vphantom{1}}\,(n).
\end{align*}

Finally define an operator $U_n$ acting on functions of $n$
which extends the action of $\mathbf{T}$ on~$\vv{E}_{\pm n}$.
This operator $U$ can easily be given explicitly by using 
\eqref{20} and \eqref{7}. Then
\begin{equation*}
U_n\Big(\big(\vv{f}\big){}\hat{\vphantom{1}}\,(n)\Big)
=\Big(\mathbf{T}\vv{f}\Big){}\hat{\vphantom{1}}\,(n).
\end{equation*}
So, in view of \eqref{132} we have settled that
\begin{equation*}
\DAHA_{a,b,c,d;q}\lan U_n,N_n,M_n\ran\simeq
\DAHA_{a,b,c,d;q}\lan T_1,Y,T_1^{-1}Z^{-1}T_1\ran
\simeq\DAHA_{a,b,c,d;q}^{\rm opp}\lan T_1,Y,Z^{-1}\ran.
\end{equation*}
We have achieved this by working with 2-vector-valued functions of $z$
and without using the DAHA duality, an approach quite different from
the one in \S\ref{137}.
\begin{remark}
Define the $2\times 2$ matrix valued polynomial
\begin{equation*}
\mathbf{E}_n[z]:=\bma \vv{E_n}[z]&\vv{E_{-n}}[z]\ema\quad
(n=0,1,2,\ldots)
\end{equation*}
where $\vv{E_{\pm n}}[z]$ are the 2-vector valued polynomials
given by \eqref{20}. Then, by \eqref{138},
\begin{equation*}
\int_{|z|=1} \Big(\mathbf{E}_m[z]\Big)^{\rm t}\,\mathbf{W}[z]\,
\mathbf{E_n}[z]\,
\frac{dz}{iz}=h_n\,\de_{m,n}\bma1&0\\0&1\ema\qquad(m,n=0,1,2,\ldots).
\end{equation*}
Thus we have matrix-valued orthogonal polynomials, see for instance
\cite{DPS}.
\end{remark}
\subsection{Duality for the 2D non-symmetric AW polynomials}
While many formulas turn out to be very satisfactory in the
2D presentation, this is less so for the
2D version of the duality \eqref{eq:15} for non-symmetric AW
polynomials. Here we will give a ``mixed'' duality formula,
with scalar-valued polynomials on the left side and 2-vector-valued
polynomials on the right side, since this is most suitable 
when taking limits to Continuous Dual Hahn and
Big $q$-Jacobi.
We will compare $E_n\big[z_{a,q}(m)^{-1}\big]$
and $\vv{\td E_m}\big[z_{\td a,q}(n)^{-1}\big]$ ($m,n\in\ZZ$),
where $z_{a,q}(n)$ is given by \eqref{26}, $E_n$ by \eqref{40} and
$\vv{E_n}$ by \eqref{20},
and where, as usual, $\vv{\td E_m}$ means $\vv{E_m}$ with respect to
dual parameters.
Now it follows from \eqref{eq:15}, \eqref{109} and \eqref{103}
that we have the duality
\begin{equation}
E_n[z_{a,q}(m)^{-1}]
=\bma1&\displaystyle
\frac{(1-q^{-1}abcd\,\nu_{\td a,q}(n))(1-ab\,\nu_{\td a,q}(n))}
{\nu_{\td a,q}(n)}\ema\vv{\td E_m}[z_{\td a,q}(n)^{-1}].
\label{140}
\end{equation}
On the right-hand side we have a matrix product of a row vector and
a column vector, which yields a scalar.
\section{Degenerations of the Askey-Wilson DAHA and the\\
non-symmetric AW polynomials}
\subsection{Degenerations of DAHA and duality}
\label{155}
The limits $d\to0$ from AW to Continuous dual Hahn and $d,c\to0$
from AW to Al-Salam-Chihara (see \eqref{67} and \eqref{eq:ASC})
have corresponding  DAHA degenerations which were discussed in \cite{M}.
Here we recall the main points of that study and construct other degenerations corresponding to the limits \eqref{68} and
\eqref{28} from AW to Big and Little $q$-Jacobi.
We will give the degenerations for the DAHA presentations
\eqref{110} and \eqref{75}.
Because certain rescalings have to be emphasized, we now use
the DAHA notation
$\DAHA_{a,b,c,d;q}[T_1,T_0,T_0^{-1},Z,Z^{-1}]$
in connection with \eqref{110} and
$\DAHA_{a,b,c,d;q}\lan T_1,Y,Y^{-1},Z,Z^{-1}\ran$
in connection with \eqref{75}.
\paragraph{\bf From AW to Continuous Dual Hahn}
Just as in \eqref{67} and \eqref{24} we have to take the limit as $d\to0$.
However, before taking the limit
it is convenient to introduce a rescaling $T_0':=q^{-1}cdT_0^{-1}$
in \eqref{110} and a rescaling $Y':=q^{-1}cdY^{-1}$ in \eqref{75} before
letting $d\to0$. Then \eqref{110} can be equivalently written as
\begin{equation*}
\begin{split}
(T_1+ab)(T_1+1)=0,\\
T_0+T_0'+1+q^{-1}cd=0,\\
(aZ^{-1}T_1^{-1}+1)(bZ^{-1}T_1^{-1}+1)=0,\\
qZ^{-1}T_0+T_0'Z+(c+d)=0,\\
T_0T_0'=q^{-1}cd=T_0'T_0,\quad ZZ^{-1}=Z^{-1}Z=1.
\end{split}
\end{equation*}
In the limit for $d\to0$ we get the algebra
\begin{equation*}
\DAHA_{a,b,c;q}^{\rm CDqH}[T_1,T_0,T_0',Z,Z^{-1}]:=
\lim_{d\to0}\DAHA_{a,b,c,d;q}
[T_1,T_0,qc^{-1}d^{-1}T_0',Z,Z^{-1}]
\end{equation*}
with generators $T_1,T_0,T_0',Z,Z^{-1}$ and relations
\begin{align*}
(T_1+ab)(T_1+1)=0,\\
\noalign{\allowbreak}
T_0+T_0'+1=0,\\
\noalign{\allowbreak}
(aZ^{-1}T_1^{-1}+1)(bZ^{-1}T_1^{-1}+1)=0,\\
\noalign{\allowbreak}
qZ^{-1}T_0+T_0'Z+c=0,\\
\noalign{\allowbreak}
T_0T_0'=0=T_0'T_0,\quad ZZ^{-1}=Z^{-1}Z=1.
\end{align*}
Since $T_0'=-T_0-1$ by the second relation, this may be
substituted in the other relations in \eqref{75d} which involve $T_0'$,
after which $T_0'$ can be dropped as a generator.
The resulting relations are
\begin{equation}
\begin{split}
(T_1+ab)(T_1+1)=0,\\
(aZ^{-1}T_1^{-1}+1)(bZ^{-1}T_1^{-1}+1)=0,\\
qZ^{-1}T_0-T_0Z-Z+c=0,\\
T_0(T_0+1)=0,\quad ZZ^{-1}=Z^{-1}Z=1.
\end{split}
\label{75d}
\end{equation}
The presentation \eqref{75d} is the same as for the algebra
$\FSH_V$ in \cite[(1.6)--(1.10)]{M}.
\par
With $Y':=q^{-1}cdY^{-1}$ the relations
\eqref{75} can be equivalently written as
\begin{equation*}
\begin{split}
(T_1+ab)(T_1+1)=0,\\
T_1^{-1}Y+Y'T_1+1+q^{-1}cd=0,\\
(aZ^{-1}T_1^{-1}+1)(bZ^{-1}T_1^{-1}+1)=0,\\
qZ^{-1}T_1^{-1}Y+Y'T_1Z+c+d=0,\\
YY'=q^{-1}cd=Y'Y,\quad ZZ^{-1}=1=Z^{-1}Z.\\
\end{split}
\end{equation*}
In the limit for $d\to0$ we get the algebra
\begin{equation*}
\DAHA_{a,b,c;q}^{\rm CDqH}\lan T_1,Y,Y',Z,Z^{-1}\ran:=
\lim_{d\to0}\DAHA_{a,b,c,d;q}
\lan T_1,Y,qc^{-1}d^{-1}Y',Z,Z^{-1}\ran
\end{equation*}
with generators $T_1,Y,Y',Z,Z^{-1}$ and relations
\begin{equation}
\begin{split}
(T_1+ab)(T_1+1)=0,\\
T_1^{-1}Y+Y'T_1+1=0,\\
(aZ^{-1}T_1^{-1}+1)(bZ^{-1}T_1^{-1}+1)=0,\\
qZ^{-1}T_1^{-1}Y+Y'T_1Z+c=0,\\
YY'=0=Y'Y,\quad ZZ^{-1}=1=Z^{-1}Z.
\label{112}
\end{split}
\end{equation}
If we replace in \eqref{112} $Z,Z^{-1}$ by $X,X^{-1}$ and
$Y'$ by $-Z$ then we recover the relations given in
\cite[Proof of Lemma 2.3]{M3} for the algebra $\FSH_V$.
\par
Similarly as for \eqref{75d} we may rewrite the second relation
in \eqref{112} as $Y':=-T_1^{-1}YT_1^{-1}-T_1^{-1}$,
substitute this in the other relations in \eqref{112}
which involve $Y'$, and then remove $Y'$ as a generator:
\begin{equation}
\begin{split}
(T_1+ab)(T_1+1)=0,\\
T_1^{-1}Y+Y'T_1+1=0,\\
(aZ^{-1}T_1^{-1}+1)(bZ^{-1}T_1^{-1}+1)=0,\\
qZ^{-1}T_1^{-1}Y-T_1^{-1}YZ-Z+c=0,\\
YT_1^{-1}Y+c=0,\quad ZZ^{-1}=1=Z^{-1}Z.
\end{split}
\label{151}
\end{equation}
\paragraph{\bf From AW to Big \boldmath$q$-Jacobi}
Just as \eqref{27} we have to
replace the substitution \eqref{68} 
and then take the limit for $\la\to0$.
Because of the way $z$ transforms in \eqref{68},
before taking the limit it is necessary to rescale
$X:=\la Z$, $X':=\la Z^{-1}$ in \eqref{110} and \eqref{75}.
After these substitutions relations \eqref{110} can be equivalently written
as
\begin{equation*}
\begin{split}
(T_1+q a)(T_1+1)=0,\\
(T_0+ b)(T_0+1)=0,\\
T_1X+q a X'T_1^{-1}+(\lambda^2+q a)=0,\\
b T_0^{-1}X+q X'T_0+q c+b/c\lambda=0,\\
XX'=\la^2=X'X.
\end{split}
\end{equation*}
In the limit for $\la\to0$ we get the algebra
\begin{equation*}
\DAHA_{a,b,c;q}^{\rm BqJ}[T_1,T_0,T_0^{-1},X,X']:=
\lim_{\la\to0}\DAHA_{\la,q a\la^{-1}, q c\la^{-1}, b c^{-1}\la;q}
[T_1,T_0,T_0^{-1},\la^{-1}X,\la^{-1}X']
\end{equation*}
with generators $T_1,T_0,X,X'$ and relations
\begin{equation}
\begin{split}
(T_1+q a)(T_1+1)=0,\\
(T_0+ b)(T_0+1)=0,\\
T_1X+q a X'T_1^{-1} +q a=0,\\
b T_0^{-1}X+q X'T_0+q c=0,\\
XX'=0=X'X.
\end{split}
\label{Hgamma}
\end{equation}
This algebra can be seen to be equivalent with the algebra
$\FSH_V^\gamma$ in \cite[(3.110)--(3.114)]{M} if we
replace there $a,b,c$ by $-a^2/q,b c/q, c q$.
\par
With  $Z:=\la^{-1}X$ and $Z^{-1}:=\la^{-1}X'$ the relations
\eqref{75} can be equivalently written as
\begin{equation*}
\begin{split}
(T_1+qa)(T_1+1)=0,\\
(T_1^{-1}Y+b)(T_1^{-1}Y+1)=0,\\
T_1X+q a X'T_1^{-1}+(\lambda^2+q a)=0,\\
b Y^{-1}T_1 X+q X'T_1^{-1}Y+q c+b/c\lambda=0,\\
XX'=\la^2=X'X.
\end{split}
\end{equation*}
In the limit for $\la\to0$ we get the algebra
\[
\DAHA_{a,b,c;q}^{\rm BqJ}\lan T_1,Y,Y^{-1},X,X'\ran=\lim_{\la\to0}
\DAHA_{\la,q a\la^{-1}, q c\la^{-1}, b c^{-1}\la;q}\lan
T_1,Y,Y^{-1},\la^{-1}X,\la^{-1}X'\ran
\]
with generators $T_1,Y,X,X'$ and with relations
\begin{equation}
\begin{split}
(T_1+qa)(T_1+1)=0,\\
(T_1^{-1}Y+b)(T_1^{-1}Y+1)=0,\\
T_1X+q a X'T_1^{-1}+q a=0,\\
b Y^{-1}T_1 X+q X'T_1^{-1}Y+q c=0,\\
XX'=0=X'X.
\end{split}
\label{113}
\end{equation}
\paragraph{\bf Duality}
Now recall the anti-isometric
dualities \eqref{115} and \eqref{104}
which read in extended notation as
\begin{equation}
\DAHA_{a,b,c,d;q}[T_1,T_0,T_0^{-1},Z,Z^{-1}]\simeq
\DAHA_{\td a,\td b,\td c,\td d;q}
[T_1,aZ^{-1}T_1^{-1},a^{-1}T_1Z,\td a T_0^{-1}T_1^{-1},
\td a^{-1}T_1T_0]
\label{116}
\end{equation}
and
\begin{equation}
\DAHA_{a,b,c,d;q}\lan T_1,Y,Y^{-1},Z,Z^{-1}\ran\simeq
\DAHA_{\td a,\td b,\td c,\td d;q}\lan T_1,aZ^{-1},a^{-1}Z,\td a Y^{-1},
\td a^{-1} Y\ran.
\label{117}
\end{equation}
\par
In \eqref{116} substitute \eqref{81} and $T_0^{-1}=ab\la^{-2}T_0'$:
\begin{multline*}
\DAHA_{a,b,c,qa^{-1}b^{-1}c^{-1}\la^2;q}[T_1,T_0,ab\la^{-2}T_0',Z,Z^{-1}]
\simeq\\
\DAHA_{\la,ab\la^{-1},ac\la^{-1},qb^{-1}c^{-1}\la;q}
[T_1,aZ^{-1}T_1^{-1},a^{-1}T_1Z,ab\la^{-1}T_0'T_1^{-1},\la^{-1}T_1T_0]
\end{multline*}
In the limit for $\la\to0$ we get
\begin{equation}
\DAHA_{a,b,c;q}^{\rm CDqH}[T_1,T_0,T_0',Z,Z^{-1}]\simeq
\DAHA_{q^{-1}ab ,ab^{-1}, q^{-1}ac;q}^{\rm BqJ}
[T_1,aZ^{-1}T_1^{-1},a^{-1}T_1Z,abT_0'T_1^{-1},T_1T_0]
\label{119}
\end{equation}
or equivalently
\begin{multline}
\DAHA_{a,b,c;q}^{\rm CDqH}[T_1,T_1^{-1}X',a^{-1}b^{-1}XT_1,
aT_1^{-1}T_0^{-1},a^{-1}T_0T_1]\simeq\\
\DAHA_{q^{-1}ab ,ab^{-1}, q^{-1}ac;q}^{\rm BqJ}[T_1,T_0,T_0^{-1},X,X'].
\label{120}
\end{multline}
These anti-isometric
dualities can also be verified directly by comparing
\eqref{75d} and \eqref{Hgamma}.
They were observed in \cite[(3.115)]{M} as an isometry from
$\FSH_V$ to $\FSH_V^\gamma$.
\par
In \eqref{117} substitute \eqref{81} and $Y^{-1}=ab\la^{-2}Y'$:
\begin{multline*}
\DAHA_{a,b,c,qa^{-1}b^{-1}c^{-1}\la^2;q}\lan T_1,Y,
ab\la^{-2}Y',Z,Z^{-1}\ran
\simeq\\
\DAHA_{\la,ab\la^{-1},ac\la^{-1},qb^{-1}c^{-1}\la;q}\lan T_1,aZ^{-1},a^{-1}Z,
ab\la^{-1} Y',\la^{-1} Y\ran.
\end{multline*}
In the limit for $\la\to0$ we get
\begin{equation}
\DAHA_{a,b,c;q}^{\rm CDqH}\lan T_1,Y,Y',Z,Z^{-1}\ran\simeq
\DAHA_{q^{-1}ab ,ab^{-1}, q^{-1}ac;q}^{\rm BqJ}
\lan T_1,aZ^{-1},a^{-1}Z,abY',Y\ran
\label{121}
\end{equation}
or equivalently
\begin{equation}
\DAHA_{a,b,c;q}^{\rm CDqH}\lan T_1,X',a^{-1}b^{-1}X,aY^{-1},
a^{-1}Y\ran\simeq
\DAHA_{q^{-1}ab ,ab^{-1}, q^{-1}ac;q}^{\rm BqJ}\lan T_1,Y,Y^{-1},X,X'\ran.
\label{122}
\end{equation}
These anti-isometric
dualities can also be verified directly by comparing
\eqref{112} and \eqref{113}.
\paragraph{\bf From Continuous Dual \boldmath$q$-Hahn to
Al-Salam-Chihara and from Big to Little \boldmath$q$-Jacobi}
Just as in \eqref{eq:ASC}, \eqref{28} and \S\ref{118} we have to let
$c\to0$ in order to arrive from Continuous Dual $q$-Hahn to
Al-Salam-Chihara and from Big $q$-Jacobi to Little $q$-Jacobi:
\begin{align*}
\DAHA_{a,b;q}^{\rm ASC}[T_1,T_0,Z,Z^{-1}]&:=
\lim_{c\to0}\DAHA_{a,b,c;q}^{\rm CDqH}[T_1,T_0,T_0',Z,Z^{-1}],\\
\DAHA_{a,b;q}^{\rm ASC}\lan T_1,Y,Y',Z,Z^{-1}\ran&:=
\lim_{c\to0}\DAHA_{a,b,c;q}^{\rm CDqH}\lan T_1,Y,Y',Z,Z^{-1}\ran,\\
\DAHA_{a,b;q}^{\rm LqJ}[T_1,T_0,T_0^{-1},X,X']&:=
\lim_{c\to0}\DAHA_{a,b,c;q}^{\rm BqJ}[T_1,T_0,T_0^{-1},X,X'],\\
\DAHA_{a,b;q}^{\rm LqJ}\lan T_1,Y,Y^{-1},X,X'\ran&:=
\lim_{c\to0}\DAHA_{a,b,c;q}^{\rm BqJ}\lan T_1,Y,Y^{-1},X,X'\ran.
\end{align*}
The corresponding formulas for the relations and for the dualities
can be obtained by putting $c=0$ in   
\eqref{75d}--\eqref{113} and
\eqref{119}--\eqref{122}.
The algebra $\DAHA_{a,b;q}^{\rm ASC}[T_1,T_0,Z,Z^{-1}]$
equals the algebra $\FSH_{III}^\be$ in \cite[Remark 6.17]{M}.
\paragraph{\bf From Continuous Dual \boldmath$q$-Hahn to
AW $q$-Bessel and from Big to Little $q$-Jacobi with $c\to\iy$}
Corresponding to the limit \eqref{149} we should let $c\to\iy$ in
\eqref{75d}. 
However, in order to get meaningful limit relations we first have to
rescale $T_0=c\td T_0$.
Then we obtain
\begin{equation*}
\DAHA_{a,b;q}^{\rm AWqB}[T_1,\td T_0,,Z,Z^{-1}]:=
\lim_{c\to\iy}\DAHA_{a,b,c;q}^{\rm CDqH}
[T_1,c\td T_0,-c\td T_0-1,Z,Z^{-1}].
\end{equation*}
with relations
\begin{equation}
\begin{split}
(T_1+ab)(T_1+1)=0,\\
(aZ^{-1}T_1^{-1}+1)(bZ^{-1}T_1^{-1}+1)=0,\\
qZ^{-1}\td T_0-\td T_0Z+1=0,\\
\td T_0^2=0,\quad ZZ^{-1}=Z^{-1}Z=1.
\end{split}
\label{150}
\end{equation}
If we compare the relations \eqref{150} with the relations \eqref{75d}
for $c=0$ then we see that they are equivalent under the substitution
$T_0=-q\td T_0Z^{-1}-1$. Thus
\begin{equation*}
\DAHA_{a,b;q}^{\rm AWqB}[T_1,\td T_0,Z,Z^{-1}]\simeq
\DAHA_{a,b;q}^{\rm ASC}[T_1,-q\td T_0Z^{-1}-1,Z,Z^{-1}].
\end{equation*}
Similar results can be formulated in connection with the limit
for $c\to\iy$ of relations \eqref{151}.
The algebra $\DAHA_{a,b;q}^{\rm AWqB}[T_1,\td T_0,Z,Z^{-1}]$
equals the algebra $\FSH_{III}$ in
\cite[(1.16)--(1.20)]{M} and \cite[(1.5)--(1.8)]{M3}.
It can be recognized as a so-called nil-DAHA \cite[Remark 8.4]{LT}.
The above correspondence between $\DAHA^{\rm ASC}$ and
$\DAHA^{\rm AWqB}$ was earlier given in
\cite[Remark 6.17]{M}.

Corresponding to the limit  \eqref{eq:l1} we should let
$c\to\iy$ in \eqref{Hgamma} and \eqref{113}.
However, in order to get meaningful limit relations we first have to rescale
$X=c\td X$, $X'=c\td X'$. Then we obtain
\begin{align*}
\wt\DAHA_{a,b;q}^{\rm LqJ}[T_1,T_0,T_0^{-1},\td X,\td X']&:=
\lim_{c\to\iy}\DAHA_{a,b,c;q}^{\rm BqJ}[T_1,T_0,T_0^{-1},c\td X,c\td X'],\\
\wt\DAHA_{a,b;q}^{\rm LqJ}\lan T_1,Y,Y^{-1},\td X,\td X'\ran&:=
\lim_{c\to\iy}\DAHA_{a,b,c;q}^{\rm BqJ}\lan T_1,Y,Y^{-1},c\td X,c\td X'\ran
\end{align*}
with relations, respectively,
\begin{equation}
\begin{split}
(T_1+ab)(T_1+1)=0,\\
(T_0+ab^{-1})(T_0+1)=0,\\
T_1\td X+ab\td X'T_1^{-1}=0,\\
ab^{-1}T_0^{-1}\td X+q\td X'T_0+a=0,\\
\td X\td X'=0=\td X'\td X,
\end{split}
\label{146}
\end{equation}
and
\begin{equation}
\begin{split}
(T_1+qa)(T_1+1)=0,\\
(T_1^{-1}Y+ab^{-1})(T_1^{-1}Y+1)=0,\\
T_1\td X+ab\td X'T_1^{-1}=0,\\
aY^{-1}T_1\td X+qb\td X'T_1^{-1}Y+ab=0,\\
\td X\td X'=0=\td X'\td X.
\end{split}
\label{147}
\end{equation}
If we compare the relations \eqref{147} with the relations \eqref{113}
for $c=0$ then we see that they are equivalent under the substitutions
$\td X=-XY^{-1}$, $ \td X'=-qa^{-2} YX'$.
Thus
\begin{equation*}
\wt\DAHA_{a,b;q}^{\rm LqJ}\lan T_1,Y,Y^{-1},-XY^{-1},-qa^{-2} YX'\ran
\simeq
\DAHA_{a,b;q}^{\rm LqJ}\lan T_1,Y,Y^{-1},X,X'\ran.
\end{equation*}
\subsection{Non-symmetric continuous dual $q$-Hahn polynomials
and Al-Salam-Chihara polynomials}
\label{154}
The non-symmetric versions of the continuous
dual $q$-Hahn can be obtained by setting $d=0$ in \eqref{40}, see \cite{M3}:
\begin{equation*}
E_n[z;a,b,c\,|\,q]:=\lim_{d\to0}E_n[z;a,b,c,d\,|\,q].
\end{equation*}
Then
\begin{equation}
\begin{split}
&E_n[z;a,b,c\,|\,q]=R_n[z;a,b,c\,|\,q]
-\frac{q^{1-n}(1-q^n)}
{(1-qab)(1-ab)(1-ac)}\\
&\qquad\times az^{-1} (1-az)(1-bz) R_{n-1}[z;qa,qb,c\,|\,q]\qquad
(n=0,1,2,\ldots),
\sLP
&E_{-n}[z;a,b,c\,|\,q]:=R_n[z;a,b,c\,|\,q]
-\frac{q^{1-n}(1-q^n ab)}
{(1-qab)(1-ab)(1-ac)}\\
&\qquad\times b^{-1}z^{-1} (1-az)(1-bz) R_{n-1}[z;qa,qb,c\,|\,q]\qquad
(n=1,2,\ldots),
\end{split}
\label{139}
\end{equation}
where $(1-q^n)E_{n-1}:=0$ for $n=0$.
\par
Corresponding to the presentation \eqref{112} of the corresponding DAHA
it is sufficient to deal with the generators $T_1,Y,Y'$ and $Z$
in its basic representation. For $T_1,Y$ and $Z$ put $d=0$ in
their formulas in \S\ref{143}.  This does not change the formulas \eqref{14}
for $Z$ and \eqref{15} for $T_1$. Formula \eqref{34} for $Y$ becomes
\begin{align*}
(Y f)[z]&=
\frac{z\big(1+ab-(a+ b)z\big)(c-z)}{(1-z^2)(q-z^2)}\,(f[z]-f[q z^{-1}])\\
&\quad+ \frac{(1-a z)(1-b z)(1-c z)}{(1-z^2)(1-q z^2)}\,(f[q z]- f[z^{-1}]).
\end{align*}
This was also given in \cite[Proof of Lemma 2.4]{M3}.
The formula for $Y'$ is obtained by putting $Y':=q^{-1}cdY^{-1}$
with $Y^{-1}$ given by \eqref{144} and then putting $d=0$:
\begin{align*}
&(Y'f)[z]=
\frac{z \big(a+b-(1+ab)z\big)(q-cz)}
{ab(1-z^2)(q-z^2)}\,f[z]\\
&\qquad\qquad\quad-\frac{z(aq-z)(bq-z)(c-z)}{ab(q-z^2)(q^2-z^2)}\,f[q^{-1}z]
+\frac{(1-a z)(1-b z)(q-cz)}
{ab(1-z^2)(q-z^2)}\,f[z^{-1}]\\
&\qquad\qquad\quad-\frac{qz\big((a+b)z-q(1+ab)\big)(c-z)}
{ab(q-z^2)(q^2-z^2)}\,f[qz^{-1}].
\end{align*}
Then, from \eqref{25} and also using the definition of $Y'$ we obtain
the eigenvalue equations
\begin{equation}
\begin{split}
YE_n&=0\qquad\qquad\qquad(n=0,1,2,\ldots),
\\
YE_{-n}&=q^{-n}\,E_{-n}\qquad\quad(n=1,2,\ldots),
\end{split}
\label{eq:eigenY}
\end{equation}
\begin{equation}
\begin{split}
Y'E_n&=q^{-n}a^{-1}b^{-1}\,E_n\quad(n=0,1,2,\ldots),
\\
Y'E_{-n}&=0\qquad\qquad\qquad\;\;(n=1,2,\ldots).
\end{split}
\label{145}
\end{equation}
Formulas \eqref{eq:eigenY} and \eqref{145} were earlier given in
\cite[Lemma 2.4]{M3}.
\par
As for the recurrence relation \eqref{101} with $M_n$ given by \eqref{102}
involving formula \eqref{103} for $\nu_{\td a ,q}(n)$, one can see from
\eqref{102} and \eqref{103} that the limit of $M_n$ for $d\to0$ exists,
where one has to distinguish between the cases $n\ge0$ and $n<0$.
We do not give the explicit formulas here.

Similar results for non-symmetric Al-Salam-Chihara polynomials
will simply follow by putting $c=0$ in the above formulas.
\subsection{Non-symmetric big and little $q$-Jacobi polynomials}
\label{153}
When taking the limit \eqref{68} to big $q$-Jacobi (and consequently
the limits \eqref{eq:l1}, \eqref{eq:l2} to little $q$-Jacobi), one
produces true polynomials rather than Laurent ones. By taking the same
limits of the non-symmetric AW polynomials \eqref{40}, one obtains
families of polynomials that are no longer functionally
independent. To overcome this difficulty we need to deal with the
$2$-D non-symmetric AW polynomials \eqref{20} and take limits of those.
Thus define the non-symmetric big $q$-Jacobi polynomials
\begin{equation}
\vv{E_n}(x;a,b,c;q):=
\lim_{\la\to0}\vv{E_n}
[\la^{-1}x;\la,qa\la^{-1},qc\la^{-1},bc^{-1}\la\,|\,q]\qquad(n\in\ZZ).
\label{158}
\end{equation}
Then
\begin{equation}
\begin{split}
\vv{E_n}(x)&=
\bma P_n(x;a,b,c\mid q)\sLP
\displaystyle-\frac{q^{1-n}(1-q^n)(1-q^{n}b)}
{(1-qa)(1-q^2a)(1-qc)}\,
P_{n-1}(q x,q^2 a,b,q c;q)\ema\quad
(n\ge0),\\
\vv{E_{-n}}(x)&=
\bma P_n(x;a,b,c\mid q)\sLP
\displaystyle-\frac{q^{-n}(1-q^{n+1} a)(1-q^{n+1}ab)}
{a (1-qa)(1-q^2 a)(1-qc)}  P_{n-1}(q x,q^2 a,b,q c;q) \ema
\quad(n\ge1),
\end{split}      
\label{20-BqJ}
\end{equation}
where $(1-q^n)P_{n-1}:=0$ for $n=0$.

We will deduce the  basic representation of $\DAHA_{a,b,c;q}^{\rm BqJ}$
by $2\times 2$ matrix-valued operators
from the one for $\DAHA$. We need to impose the substitution
\begin{equation}\label{eq:sub-AWBqJ}
sub=\{{z\to \la^{-1}x, a\to \la,b\to qa\la^{-1},c\to
qc\la^{-1},d\to bc^{-1}\la}\},
\end{equation}
defined in \eqref{68}, in the 2D realization of the basic
representation of $\DAHA$. However, when taking the limit 
as $\lambda\to 0$, we see that to obtain well defined matrix operators
we need to conjugate all operators by the diagonal matrix with entries
$1,\frac{1}{\lambda}$. Moreover, because $z\to \la^{-1}x$, we need to
multiply $Z$ by $\lambda$. Therefore, we introduce the
following rescalings of \eqref{5}, \eqref{6}
and \eqref{7}:
\begin{align}
\wt {\bf Z}:&=\la\bma 1&0\\0&1/\la\ema  {\bf Z}_{sub} \bma 1&0\\0&\la\ema,
\label{52}\\
(1-qa)x \wt Z_{11}&=
x^2+\la^2-x (\la^2+q a),
\nonu\\
 (1-qa)x \wt Z_{22}&=
\la^2 x-q a(\la^2+x^2-x),
\nonu\\
(1-qa)x^2\wt Z_{12}&=(x-\la^2)(q a x-\lambda^2)(x-1)
(x-qa ),
\nonu\\
(1-qa)\wt Z_{21}&=-1;
\nonu\\
\noalign{\allowbreak}
\wt{ {\bf Z}^{-1}}:&=\la\bma 1&0\\0&1/\la\ema  {\bf Z}^{-1}_{sub}
\bma 1&0\\0&\la \ema
=\bma \wt Z_{22}&-\wt Z_{12}\sLP-\wt Z_{21}&\wt Z_{11}\ema;
\label{54}\\
\noalign{\allowbreak}
\wt  {\bf T}_1:&=\bma 1&0\\0&1/\la\ema  {\bf T}_{1_{sub}} \bma 1&0\\0&\la\ema
=\bma -q a &0\\0&-1\ema;
\\
\noalign{\allowbreak}
\wt  {\bf Y}:&=\bma 1&0\\0&1/\la\ema  {\bf Y}_{sub} \bma 1&0\\0&\la\ema,
\label{40-0}
\end{align}
where $ {\bf Z}^{\pm 1}_{sub},  {\bf T}_{1_{sub}}$ and  ${\bf Y}_{sub}$
denote the operators in which we have performed the substitution
\eqref{eq:sub-AWBqJ}.
Denote the limits for $\la\to0$ by ${\bf X},{\bf X}',{\bf T}_1,{\bf Y}$,
respectively. Then we obtain:
\[
{\bf X}=\frac{1}{q a-1}\bma q a-x&q a(x-1)(q a-x)\\
1&q a(x-1)\ema,
\]
\[
{\bf X}'=\frac{1}{q a-1}\bma q a(x-1)&-q a(x-1)(q a-x)\\
-1&q a-x\ema,\quad
{\bf T}_1=\bma-q a&0\\0&-1\ema.
\]

For ${\bf Y}$, which we have not given explicitly, we obtain
from \eqref{21} that
\begin{equation}
\begin{split}
{\bf Y}\vv{E_{-n}}(x)&=q^{-n}\,\vv{E_{-n}}(x)\qquad\qquad(n=1,2,\ldots),\\
{\bf Y}\vv{E_{n}}(x)&=q^{n+1}a b\,\vv{E_{n}}(x)\qquad\;(n=0,1,2,\ldots).
\end{split}
\label{21-BqJ}
\end{equation}

Finally,  the analogue of \eqref{134} is:
\begin{equation}
\tilde M_n\Big(\vv{E_n}(x)\Big)=\mathbf{X}' \vv{E_n}(x),
\label{134b}
\end{equation}
where
\begin{equation}
\begin{split}
&\tilde M_n(g(n))
=\frac{\mu_{\td a ,q}(n)(a b \mu_{\td a ,q}(n)-c)
(a(1+b)q\mu_{\td a ,q}(n)-1-q a)}
{(a b \mu_{\td a ,q}(n)^2-1)(a b q \mu_{\td a ,q}(n)^2-1)}
\left(g(n)-g(-1-n)\right)\\
&\qquad\qquad\quad+\frac{(1-a q \mu_{\td a ,q}(n))(1-a b q \mu_{\td a ,q}(n))
(1-c q \mu_{\td a ,q}(n))}{(a b q^2 \mu_{\td a ,q}(n)^2-1)
(a b q \mu_{\td a ,q}(n)^2-1)} (g(n-1)-g(-n))\\
\end{split}
\end{equation}
and:
\begin{equation}
\mu_{\td a ,q}(n):=
\begin{cases}
(ab q^{1+n})^{-1},&n\ge0,\\
q^{-n},&n<0.
\end{cases}
\label{103b}
\end{equation}
This is proved straight from  \eqref{134} by substitution.

Similar results for non-symmetric little $q$-Jacobi polynomials
will simply follow by putting $c=0$ in the above formulas.
\begin{remark}
\label{161}
Clearly, from \eqref{68} there is a symmetry
$P_n(x;a,b,c)=P_n(x;c,ab/c,a)$.
{Hence, from \eqref{20-BqJ},
\small\begin{equation}
\begin{split}
\vv{E_n}(x;c,ab/c,a;q)&=
\bma P_n(x;a,b,c\mid q)\sLP
\displaystyle-\frac{q^{1-n}(1-q^n)(c-q^n ab)}
{c(1-qa)(1-qc)(1-q^2c)}\,
P_{n-1}(q x,qa,qb,q^2 c;q)\ema,\\
\vv{E_{-n}}(x;c,ab/c,a;q)&=
\bma P_n(x;a,b,c\mid q)\sLP
\displaystyle-\frac{q^{-n}(1-q^{n+1} c)(1-q^{n+1}ab)}
{c(1-qa)(1-qc)(1-q^2 c)}  P_{n-1}(q x,qa,qb,q^2 c;q) \ema
\end{split}
\label{159}
\end{equation}}

\noindent
for $n\ge0$ respectively $n>0$.
Note that the $q$-shifts in the parameters of the big $q$-Jacobi
polynomials occurring in the second coordinate in \eqref{159}
are different from the ones in \eqref{20-BqJ}.
The $q$-shifts in \eqref{159} are more in agreement with the
$q$-shifts in the vector-valued little $q$-Jacobi polynomials
discussed in \cite[\S6]{KB}.
In fact, the limit for $c\to0$ of
$\bma 1&0\\0&c\ema \vv{E_n}(x;c,ab/c,a;q)$ essentially gives
these polynomials \cite[(6.4), (6.5)]{KB}.

The polynomials \eqref{159} will also
be eigenfunctions of the $Y$ operator in the basic representation
of $\DAHA_{c,ab/c,a;q}^{\rm BqJ}$. A limit for $c\to0$ will be possible
in this eigenvalue equation (a little $q$-Jacobi case).
However, it is not clear at all if some decent algebra will result from
taking the limit
for $c\to0$ of $\DAHA_{c,ab/c,a;q}^{\rm BqJ}$.
\end{remark}
\subsection{Duality between degenerate cases of non-symmetric
AW polynomials}
By comparing \eqref{139} and \eqref{20-BqJ} we obtain for
$m,n\in\ZZ$ that
\begin{equation}
E_n(z_{a,q}(m)^{-1};a,b,c\,|\,q)
=\bma1&\mu_{ab,q}(n)\ema
\vv{E_m}(q^{-n};q^{-1}ab,ab^{-1},q^{-1}ac;q),
\label{141}
\end{equation}
where
\begin{equation*}
\begin{split}
\mu_{ab,q}(n)&:=abq^{-n}(1-q^n)\qquad(n=0,1,2,\ldots),\\
\mu_{ab,q}(-n)&:=q^{-n}(1-q^nab)\qquad(n=1,2,\ldots),
\end{split}
\end{equation*}
and $z_{a,q}(m)$ is given by \eqref{26}.
As in \eqref{140}, formula \eqref{141} has a matrix multiplication
of a row vector with a column vector on the right.
Formula \eqref{141} is also a limit case of \eqref{140}.

The duality \eqref{141} extends to a duality for the operators
acting on both sides as given in \S\S\ref{154}, \ref{153}.
These will come from the dualities of the degenerate DAHA's,
given in \S\ref{155}, in their basic representations.
Furthermore, everything can be specialized to the next level
in the $q$-Askey scheme by putting $c=0$.
\section{Summary of other related work and further perspective}
\label{162}
Necessarily, given the limited size of a journal article, we had to restrict
ourselves in the choice of material. This has resulted in a treatment
of material related to the part of the $q$-Askey
scheme depicted in Figure \ref{fig:1}. A more comprehensive study
of degenerate DAHA's associated with the ($q$-)Askey scheme would
have made links with cases already studied in the literature (often also
in higher rank):
\begin{itemize}
\item
Non-symmetric dual $q$-Krawtchouk polynomials are related to
Cherednik's one-dimenional nil-DAHA, see \cite{LT}. As we already
mentioned, this nil-DAHA is very close to the degenerate DAHA's
for Al-Salam-Chihara and little $q$-Jacobi considered in \S\ref{155}.
\item
Non-symmetric Wilson polynomials are related to a degenerate DAHA
considered in \cite{Groe2} and \cite{GVZ}.
\item
Non-symmetric Jacobi polynomials are related to
the case $n=1$ of
the dDAHA (degenerate DAHA) of type $BC_n$
considered in \cite[\S3.1]{EFM}.
\item
Non-symmetric Bessel functions (limit cases of non-symmetric Jacobi
polynomials) are related to a rational Cherednik algebra \cite{BEG}
of rank one.
\end{itemize}
Furthermore, we did not treat the
finite and infinite discrete families, with the $q$-Racah polynomials
on top, see also Remark \ref{160}. From the point of view of duality
such families were classified by Leonard \cite{L}. A further discussion
of the families arising from Leonard's classiifcation was given in
\cite{BI}, including the $q\to-1$ limit to the Bannai-Ito polynomials.
Recently a lot of work on $q\to-1$ limits has been done by Vinet, Zhedanov
and coauthors. See in particular \cite{GVZ}, where the $q\to-1$ degeneration
of the Zhedanov algebra associated with the Banna-Ito polynomials is
identified with the degenerate DAHA associated with the non-symmetric
Wilson polynomials.

In a different line of development the paper \cite{TVZ} introduced
analogues of Askey-Wilson polynomials which are orthogonal on the
unit circle, and constructed a DAHA associated with them.

An evident perspective for further work is to describe a full
($q$-)Askey scheme of non-symmetric orthogonal polynomials and the
associated degenerate DAHA's. Important questions here will be when
it is necessary to work with vector-valued polynomials rather than
Laurent polynomials, whether the orthogonality relations \eqref{138}
for vector-valued AW survive in the limit cases (for a few special cases
positively answered in \cite{KB}), and what the consequences are
when limits of nonsymmetric AW are taken with permuted parameters
(see Remark \ref{161}). The ``non-symmetric''
($q$-)Askey scheme should also be extended to non-polynomial cases
(cf.~\cite{KS}). All such work should finally get analogues in the higher
rank ($BC_n$) case.
\section*{Acknowledgments}
The authors are grateful to J. Stokman for useful conversations and to
the organisers of the $14$-th International Symposium on Orthogonal
Polynomials, Special Functions and Applications for inviting them to
give a talk, in particular, a plenary lecture by the second author.
We also thank Erik Koelink, Paul Terwilliger, and Alexei Zhedanov
for mentioning relevant references we had missed. Finally we thank a
referee for careful reading and for an important comment.
The research of M. Mazzocco was funded by EPSRC Research Grant
EP/P021913/1.

\end{document}